\documentclass[a4paper,11pt]{article}

\usepackage[affil-it]{authblk}
\usepackage{amsmath}
\usepackage{amsfonts}
\usepackage{hyperref}

\usepackage{amsthm}
\usepackage{amssymb}
\usepackage{graphicx}
\usepackage{tensor}
\DeclareGraphicsExtensions{.jpg,.png,.pdf,.eps,.bmp,.gif}
\newtheorem{theo}{Theorem}
\newtheorem{defi}{Definition}
\newtheorem{prop}[theo]{Proposition}
\newtheorem{lem}[theo]{Lemma}
\newtheorem{cor}[theo]{Corollary}

\newtheorem{rem}[theo]{Remark}

\title{A stroll along the gamma} \author{Benjamin
  Arras\footnote{Universit\'e de Li\`ege, Sart-Tilman, All\'ee de la
    d\'ecouverte 12, B-8000 Li\`ege, Belgium
    \texttt{barras@ulg.ac.be}; BA's research is supported by a Welcome
    Grant from ULg. } \ and Yvik Swan\footnote{Universit\'e de
    Li\`ege, Sart-Tilman, All\'ee de la d\'ecouverte 12, B-8000
    Li\`ege, Belgium \texttt{yswan@ulg.ac.be}; YS gratefully
    acknowledges support from the IAP Research Network P7/06 of the
    Belgian State (Belgian Science Policy). }}  \affil{Universit\'e de
  Li\`ege}
  
  \date{}

\begin{document}
\maketitle

\begin{abstract}
  We provide the first in-depth study of  the  ``smart path'' interpolation
  between an arbitrary probability measure and the
  gamma-$(\alpha, \lambda)$ distribution. We propose new explicit
  representation formulae for the ensuing process as well as a new
  notion of relative Fisher information with a gamma target
  distribution. We use these results to prove a differential and
  an integrated De Bruijn identity which hold under minimal
  conditions, hereby extending the classical formulae which follow from
  Bakry, Emery and Ledoux's $\Gamma$-calculus. Exploiting a 
  specific representation of the ``smart path'', we  obtain a new proof of 
  the logarithmic Sobolev inequality for the gamma law with
  $\alpha\geq 1/2$ as well as
  a new type of HSI inequality linking relative entropy, Stein discrepancy and 
  standardized Fisher information for the gamma law with $\alpha\geq 1/2$.

\


\


\end{abstract}







\section{Introduction}

Several choices of metrics are possible in order to quantify the
distance between two probability measures, $\mathbb{P}$ and
$\mathbb{Q}$ on $(\mathbb{R},\mathcal{B}(\mathbb{R}))$. A natural
choice is the so-called total variation distance between $\mathbb{P}$
and $\mathbb{Q}$ defined by:
\begin{align*}
\operatorname{TV}(\mathbb{P},\mathbb{Q})=\underset{A\in\mathcal{B}(\mathbb{R})}{\sup}\big|\mathbb{P}(A)-\mathbb{Q}(A)\big|.
\end{align*}
Assume now that $\mathbb{P}$ and $\mathbb{Q}$ are absolutely
continuous with respect to the Lebesgue measure, with densities $f$
and $g$ respectively. Then, 
a strong control of the total variation distance between $\mathbb{P}$
and $\mathbb{Q}$ can be achieved by means of the well-known Pinsker
inequality:
\begin{align*}
\operatorname{TV}(\mathbb{P},\mathbb{Q})\leq \sqrt{\dfrac{D(f\|g)}{2}},
\end{align*}
where 
\begin{align*}
D(f\|g)=\int_{\mathbb{R}}f(x)\log\bigg(\dfrac{f(x)}{g(x)}\bigg)dx
\end{align*}
is the relative entropy of $f$ with respect to $g$ (or equivalently of
$\mathbb{P}$ with respect to $\mathbb{Q}$).  Controlling
efficiently the relative entropy $D(f\|g)$ in order to quantify the
distance between $\mathbb{P}$ and $\mathbb{Q}$ is a natural
strategy which has been thoroughly addressed when $\mathbb{Q}$ is
the standard Gaussian probability measure on $\mathbb{R}$ (or on
$\mathbb{R}^n$ as well),
\begin{align*}
\mathbb{Q}(dx)=\frac{1}{\sqrt{2\pi}}\exp\bigg(-\frac{x^2}{2}\bigg)dx.
\end{align*}
More precisely, many authors have provided proofs of quantitative
versions of the central limit theorem by means of relative entropy and
other information-theoretic quantities. This line of research has its
roots in Linnik's works (\cite{Linnik}), culminated with the works of
Artstein, Ball, Barthe and Naor in \cite{ABBN,BBN}, of Johnson and
Barron in \cite{Barr2} and of Tulino and Verdu in \cite{Tulino} and to
this date is still an active and fruitful area of research
\cite{BCG1,BCG2,BCG3}. 

Controlling the relative entropy is by no means easier than the total
variation and, at the core of all the previously cited literature, stands the
celebrated \emph{De Bruijn formula} linking the relative entropy with
one further ingredient: the   Fisher information. Namely, for any real random variable $X$ with mean $0$,
unit variance and whose law $\mathbb{P}=\mathbb{P}_X$ is absolutely
continuous:
\begin{equation}\label{eq:0}
D(\mathbb{P}_X\|\mathbb{Q})=\int_0^1\dfrac{J_{st}(\sqrt{t}X+\sqrt{1-t}Z)}{2t}dt
= \int_0^1\dfrac{I(\sqrt{t}X+\sqrt{1-t}Z)-1}{2t}dt,
\end{equation}  
where $Z$ is an independent standard normal random variable,
$J_{st}(Y)=\mathbb{E}[(\rho_Y(Y)+Y)^2]$ is the relative Fisher
information of $Y$, $I(Y) = E[\rho_Y(Y)^2]$ is its Fisher information
and $\rho_Y(u)=\partial/\partial_u(\log(f_Y(u)))$ is the score
function of $Y$. As discussed in \cite{Carlen2}, identity \eqref{eq:0}
holds solely under a second moment assumption on the density of $X$
and was first derived and exploited in Barron (\cite{Barr1,Barr3}).
The first occurrence of such an identity can be traced back to
\cite{BE,BGL} albeit this time under quite stringent regularity
assumptions on the density of $X$. De Bruijn's formula \eqref{eq:0} is
also of central importance for the proof of functional inequalities
such as the logarithmic Sobolev inequality and the more recent HSI
inequality (see \cite{Gross,BE,Carlen,LNP,Bobkov} and \cite{BGL} for a
modern account on this and related topics).

Gamma limit theorems have been obtained in the context of Stein's
method on Wiener chaos (see \cite{NP1,NP2}). Influenced by the
recent paper \cite{NPS}, it is therefore natural to enquire whether entropic gamma limit 
theorems on Wiener chaos can be obtained as in the Gaussian case. A 
first obstacle to a positive answer to this question is the validity
of a general De Bruijn-type formula such as \eqref{eq:0} in the gamma case.

To our knowledge, De Bruijn-type formulae with general reference
probability measure are only known in the context of $\Gamma$-calculus
(see \cite{BE,BGL}).  Let us inspect how \eqref{eq:0} translates when
the reference measure is the gamma law. For this purpose, we introduce
some notations. We denote by $\gamma_{\alpha,\lambda}(.)$ the density
of the gamma law of parameters $(\alpha,\lambda)$ given by:
\begin{align*}
\forall u>0,\ \gamma_{\alpha,\lambda}(u)=\dfrac{\lambda^{\alpha}}{\Gamma({\alpha})}u^{\alpha-1}\exp(-\lambda u).
\end{align*}
Let $P^{(\alpha,\lambda)}_{t}$, $\mathcal{L}_{\alpha,\lambda}$ and
$\Gamma_{\alpha,\lambda}$ be the Laguerre semigroup, the Laguerre
operator and the carr\'e du champs operator naturally associated with
the gamma probability measure (which is  invariant and reversible under
the Laguerre dynamic). These operators are given respectively by the
following formulae on subsets of
$L^2\big(\gamma_{\alpha,\lambda}(u)du\big)$ (see \cite{LaguerreRef}):
\begin{align*}
&P^{(\alpha,\lambda)}_{t}(f)(x)=\dfrac{\lambda^{1-\alpha}\Gamma(\alpha)e^{\lambda t}}{e^{\lambda t}-1}\int_0^{+\infty}f(u)\bigg(\dfrac{e^{\lambda t}}{xu}\bigg)^{\frac{\alpha-1}{2}}\exp\bigg(-\frac{\lambda}{e^{\lambda t}-1}(x+u)\bigg)I_{\alpha-1}\bigg(\dfrac{2\lambda \sqrt{uxe^{\lambda t}}}{e^{\lambda t}-1}\bigg)\gamma_{\alpha,\lambda}(du),\\
&\mathcal{L}_{\alpha,\lambda}(f)(u)=u\dfrac{d^2 f}{du^2}(u)+(\alpha-\lambda u)\dfrac{d f}{du}(u),\\
&\Gamma_{\alpha,\lambda}(f)(u)=u\big(\dfrac{d f}{du}(u)\big)^2,
\end{align*}
where $I_{\alpha-1}(.)$ is the modified Bessel function of the first
kind of order $\alpha-1$. We can now write the local form of De Bruijn
identity with respect to the gamma measure as stated in  \cite[Proposition
5.2.2]{BGL}: 
\begin{align}
\forall f\geq 0,\ f\in\mathcal{D}(\mathcal{E}),\ \dfrac{d}{dt}\operatorname{Ent}_{\gamma_{\alpha,\lambda}}(P^{(\alpha,\lambda)}_t(f))=-I_{\gamma_{\alpha,\lambda}}(P^{(\alpha,\lambda)}_t(f))\label{less}, 
\end{align}
with,
\begin{align*}
&\operatorname{Ent}_{\gamma_{\alpha,\lambda}}(f)=\int_0^{+\infty}f(x)\log(f(x))\gamma_{\alpha,\lambda}(x)dx,\ I_{\gamma_{\alpha,\lambda}}(f)=\int_{0}^{+\infty}\dfrac{\Gamma_{\alpha,\lambda}(f)(x)}{f(x)}\gamma_{\alpha,\lambda}(x)dx,
\end{align*}
and 
\begin{equation*}
  f\in\mathcal{D}(\mathcal{E})\Leftrightarrow \int_0^{+\infty}\Gamma_{\alpha,\lambda}(f)(x)\gamma_{\alpha,\lambda}(x)dx<+\infty.
\end{equation*}
When $f=f_X/\gamma_{\alpha,\lambda}$ with $f_X$ the Lebesgue density
of a probability measure on $(0,+\infty)$,
$\operatorname{Ent}_{\gamma_{\alpha,\lambda}}(f)=D(\mathbb{P}_X\|\gamma_{\alpha,\lambda})$
leading to a gamma-$(\alpha, \lambda)$-specific De Bruijn identity:
\begin{equation}
  \label{eq:3}
 D(X\|\gamma_{\alpha,\lambda}) = \int_0^{\infty} J_{st,
 \gamma}(X_t) dt
\end{equation}
with $X_t$ having Lebesgue density
$\gamma_{\alpha,\lambda}(u)P^{(\alpha,\lambda)}_{t}\big(f_X/\gamma_{\alpha,\lambda}\big)(u)$
and $J_{st, \gamma}(Y) = E [ Y ( \rho_Y(Y)+ \lambda-(\alpha-1)/Y)^2]$
the resulting relative Fisher information (see Section \ref{sec-Fish}
for explanations).

This semigroup approach is not reserved towards a gamma target
probability measure and can be adapted e.g.\ to invariant measures of
second order differential operators, which also includes beta
distributions as well as families of log-concave measures as
illustrations. Identity \eqref{eq:3} is, however, not as satisfactory
as the Gaussian de Bruijn identity \eqref{eq:0} for two reasons.
\begin{enumerate}
\item \label{item:2} First, if $X$ is a Wiener chaos random variable,
  the condition
  $f_X/\gamma_{\alpha,\lambda}\in\mathcal{D}(\mathcal{E})$ is not
  testable since very little information regarding this type of random
  variable is available in the literature (see
  \cite{Jan,Nua,NP3}). Moreover, this regularity assumption upon the
  law of $X$ is a tad uncanny compared to the Gaussian case where only
  moment conditions need to hold to ensure the validity of the local
  form of the Gaussian De Bruijn identity (see
  \cite{Barr3,Barr1,Carlen2}).
\item \label{item:1} Second, and perhaps more importantly, there lacks
  an explicit representation of $X_t$ in terms of the contributions
  $X$ and $\gamma_{\alpha, \lambda}$; such a representation (in the
  Gaussian case $X_t = \sqrt{t} X + \sqrt{1-t}Z$) is indeed at the
  heart of all above cited literature and its absence in the gamma
  case dashes any hopes of developing similar techniques for a gamma
  target as those developed in the Gaussian case.
\end{enumerate}
The main results of this paper address these two issues : in Section
\ref{sec-GamInter} we present explicit stochastic representations of
the interpolation scheme between the probability measures
$\mathbb{P}_X$ and $\gamma_{\alpha,\lambda}$ along the Laguerre
dynamic under minimal assumptions on $\mathbb{P}_X$ (hereby solving
point \ref{item:1}) which we then use in Section \ref{sec-DeBruijn} to
overcome the regularity conditions necessary for the local De Bruijn
identity (hereby solving point \ref{item:2}). In particular we obtain
very general forms of local and integrated De Bruijn identity in the
gamma case for $\alpha\geq 1/2$ (see Theorem \ref{DeBruijn2} and Remark \ref{better} and
Theorem \ref{sec:de-bruijn-formula} and Remark \ref{sec:shannon},
respectively). As applications we obtain a new proof of the
logarithmic Sobolev inequality in the case $\alpha\ge 1/2$
(Proposition \ref{LSI}, Section \ref{sec-LSI}) as well as a new form
of HSI inequality for the gamma measure with $\alpha \ge 1/2$ (Theorem
\ref{HSI}, Section \ref{sec-HSI}).

 To enhance the readability of the text, the next Subsection
 \ref{sec:overview-paper} contains an intuitive description of our
 method and of our main results, as well as more insight into the
 difference between our results and those already available from the
 literature (with particular emphasis on a comparison with the
 well-trodden Gaussian case and formula \eqref{eq:0}). 

\subsection{Overview of the results}
\label{sec:overview-paper}

Let $X$ be a positive random variable. The crucial first step of this
article is to identify the correct probabilistic representation for
the smart path evolute $X_{t}$ from formula \eqref{eq:3}. Our
construction is a priori not intuitive, as it relies on an
interpolation between the target gamma distribution and a random
variable $Y = Y(X)$ defined according to the following three
stage-procedure. Fix $\tau\in (0,1)$. First draw a (strictly positive)
real number, $x$, according to $\mathbb{P}_X$; next draw a Poisson
random variable with parameter $x\lambda\tau/(1-\tau)$. Finally,
depending on the outcome, draw a gamma random variable with shape
parameter equal to the value of the Poisson random variable (with the
convention that a gamma random variable with shape parameter equal to
0 is 0 almost surely) and scale parameter equal to
$\lambda \tau/(1-\tau)$.  We denote $Y(\tau,X,\lambda)$ the resulting
random variable and define
\begin{equation}\label{eq:5}
X_{\tau}=\big(1-\tau\big)\gamma(\alpha,\lambda)+\tau Y(\tau,X,\lambda),
\end{equation}
with $\gamma(\alpha, \lambda)$ an independent gamma random variable
with density $\gamma_{\alpha,\lambda}$ (see Definition \ref{main} for
a more formal construction). 

In Section \ref{sec-GamInter} we will show that $X_{\tau}$ as defined
by \eqref{eq:5} is the gamma-$(\alpha, \lambda)$ equivalent of the
Gaussian (Ornstein-Uhlenbeck) interpolation
$X_t=\sqrt{t}X + \sqrt{1-t}Z$. Note in particular how, in
\eqref{eq:5}, the contributions of $X$ and $\gamma(\alpha, \lambda)$
are dissociated, as is desirable for applications. Moreover, when $\mathbb{P}_X$ admits a Lebesgue density, $f_X$, the random
variable $X_{\tau}$ is linked to the Laguerre dynamic through the
following formula for its density $g^{(\alpha,\lambda)}(\tau,u)$:
\begin{align*}
g^{(\alpha,\lambda)}(\tau,u)=\gamma_{\alpha,\lambda}(u)P^{(\alpha,\lambda)}_{-\frac{\log(\tau)}{\lambda}}\big(f_X/\gamma_{\alpha,\lambda}\big)(u).
\end{align*}
The construction of the random variable $X_{\tau}$ and the existence
of $g^{(\alpha,\lambda)}(\tau,u)$ do not rely on the existence of a
Lebesgue density for the initial probability measure, $\mathbb{P}_X$.
Working directly with $D(X_{\tau}\|\gamma_{\alpha,\lambda})$ will lead
to a substantial improvement of the local form of De Bruijn identity
(\ref{less}), as will be proved in Theorem \ref{DeBruijn2}, Section
\ref{sec-DeBruijn}, where we obtain an equivalent identity which holds
true solely under a moment condition for $\mathbb{P}_X$ whether a
density exists or not for the random variable $X$.

The relative Fisher information we identify for the gamma target
distribution is the same as that obtained through the semigroup
approach in \eqref{eq:3}, namely
$ J_{st, \gamma}(Y) = E [ Y ( \rho_Y(Y)+ \lambda-(\alpha-1)/Y)^2]$,
with $\rho_Y(u)$ the classical score of $Y$. Still using the notation
$\gamma(\alpha, \lambda)$ for a gamma-$(\alpha, \lambda)$ distributed
random variable we note how
$\rho_{\gamma(\alpha, \lambda)}(u) = (\alpha-1)/u - \lambda$ so that
the relative Fisher information can equivalently be rewritten
$ J_{st, \gamma}(Y) = E [ Y ( \rho_Y(Y)-\rho_{\gamma(\alpha,
  \lambda)}(Y))^2]$.
When $\alpha>1$ one can simply expound the square and apply
integration by parts to obtain
\begin{align*}
  J_{st, \gamma}(Y) & =  E [ Y
                      \rho_Y(Y)^2]  -  E [ Y  \rho_{\gamma(\alpha, \lambda)}(Y)^2]  
\end{align*} 
which holds as long as $E[Y] = \alpha/\lambda$ and
$E[1/Y] = \lambda/(\alpha-1)$. In view of this simple and intuitive
computation it might be tempting to introduce a new ``gamma-Fisher
information'' of the form $I_Y(Y) = E [ Y \rho_Y(Y)^2]$ for which the
above computation leads to the elegant fact that the relative Fisher
information decomposes into the difference of gamma-Fisher
informations, similarly as in the Gaussian case. The quantity
$I_Y(Y)$, however, suffers many flaws, the most cumbersome of which
being that it is only finite if $\alpha>1$ if
$Y\sim \gamma_{\alpha, \lambda}$. Such an assumption is not natural.
Interestingly, along the course of the proofs of Theorems
\ref{DeBruijn2} and \ref{sec:de-bruijn-formula} (integrated version of
De Bruijn identity), we will be led to introducing a new notion of
Fisher information (Definition \ref{NewFisher}):
\begin{align*}
  I^{\tau}_{\gamma}(Y)=\mathbb{E}\big[Y\big(\rho_Y(Y)+\frac{\lambda}{(1-\tau)}
  - \frac{\alpha-1}{Y}\big)^2\big]. 
\end{align*}
 Note how 
$I^{\tau}_{\gamma}(\gamma(\alpha, \lambda))= {\alpha\lambda
  \tau^2}/{(1-\tau)^2}$ is finite for every value of $\alpha>0$. We will also prove 
(Proposition \ref{FiniteFisher}) that this information satisfies the
Cramer-Rao inequality  
\begin{align*}
  J_{st,\gamma}(X_{\tau})=I^{\tau}_{\gamma}(X_{\tau})-\dfrac{\alpha\lambda\tau^2}{(1-\tau)^2}\geq 0.
\end{align*}
at all points $\tau$ along the smart path (see Section \ref{sec-Fish},
Propositions \ref{FiniteFisher} and \ref{CramerRao}) when $\mathbb{E}[X]=\alpha/\lambda$.  This corrected
and localized version of Fisher information seems to be the relevant
one for the study of Gamma comparison, and is the key to our new De
Bruijn identity for the gamma target under minimal assumptions. It
also opens the way for our two main applications described in Sections
\ref{sec-LSI} and \ref{sec-HSI}. 


As a first application, we obtain a new proof of the logarithmic
Sobolev inequality in the gamma case for $\alpha\geq 1/2$ (see
Proposition \ref{LSI}). Thanks to the Bakry-Emery criterion, this
functional inequality is known to hold for $\alpha\geq 1/2$ with a
constant independent of $\alpha$ (see \cite{Bakry}). The method of the
proof developed in Section $5$ extends the well-known Cauchy-Schwarz
argument of the Gaussian case. Namely, we obtain a new representation
for the Laguerre semigroup thanks to the following identity in law
(see Proposition \ref{Melherlike3}):
\begin{align*}
X_{\tau}\overset{\mathcal{L}}{=}(1-\tau)\gamma\big(\alpha-\frac{1}{2},\lambda\big)+\bigg(\sqrt{\tau}\sqrt{X}+\sqrt{\dfrac{1-\tau}{2\lambda}}Z\bigg)^2.
\end{align*}
This equality allows us to derive an appropriate intertwining relation which leads to a fundamental sub-commutation inequality in order to control the standardized Fisher information structure along the gamma smart path by the standardized Fisher information structure of $X$ (see Lemma \ref{MehInter} and Proposition \ref{LSI}).

The second and final application presented in this paper uses the
previous stochastic representation for the gamma smart path, we are
able to derive a new HSI inequality for the gamma law with
$\alpha\geq 1/2$ (see Theorem \ref{HSI}). HSI inequalities for
different types of probability law were introduced for the first time
in \cite{LNP}. They allow to link the relative entropy (H) and the
standardized Fisher information (I) to another kind of distance
between probability measures, namely the Stein discrepancy (S). This
Stein discrepancy is defined by means of a natural implicit quantity
from Stein approximation method, the Stein kernel (see
\cite{CPU94,Cha08,NP1,Cha09,Cha12,NP3,NPS, NPS2}). Moreover, in the Gaussian
case, it is proved in \cite{LNP} that this inequality improves upon
the classical logarithmic Sobolev inequality (see Theorem
$2.2$). Regarding the gamma case, the authors of \cite{LNP} obtain an
HSI inequality in Proposition $4.3$. Our result is different from
theirs since we do not use the same Stein kernel. Namely, for random
variables with values in $(0,+\infty)$, ours is defined via the
integration by parts identity:
\begin{align*}
\forall\phi\in\mathcal{C}^{\infty}_c\big((0,+\infty)\big),\ \mathbb{E}[(\lambda X-\alpha+\frac{1}{2})\phi(X)]=\mathbb{E}[\tau_X(X)\sqrt{X}(\partial_x^{\sigma})^*(\phi)(X)],
\end{align*}
where $(\partial_x^{\sigma})^*(.)=\frac{d}{dx}(\sqrt{x}.)$. Let us comment briefly on this choice. First of all, we note that the Laguerre operator can be rewritten using the operator $\partial_x^{\sigma}=\sqrt{x}\frac{d.}{dx}$ in the following way:
\begin{align*}
\mathcal{L}_{\alpha,\lambda}(\phi)(u)=(\partial_u^{\sigma})^2(\phi)(u)+(\alpha-\frac{1}{2}-\lambda u)\partial_u(\phi)(u).
\end{align*}
Moreover, plugging $\phi'$ in the definition of the Stein kernel, we have:
\begin{align*}
\mathbb{E}[(\lambda X-\alpha+\frac{1}{2})\phi'(X)]=\mathbb{E}[\tau_X(X)(\partial_x^{\sigma})^2(\phi)(X)].
\end{align*}
In particular, for a gamma random variable with parameters $(\alpha,\lambda)$, we have:
\begin{align*}
\mathbb{E}[(\lambda \gamma_{\alpha,\lambda}-\alpha+\frac{1}{2})\phi'(\gamma_{\alpha,\lambda})]=\mathbb{E}[(\partial_x^{\sigma})^2(\phi)(\gamma_{\alpha,\lambda})].
\end{align*}
Thus, the proximity of $\tau_X(X)$ to $1$ should indicate the
proximity (in law) of $X$ to the random variable
$\gamma_{\alpha,\lambda}$. The Stein discrepancy is then defined as
the quadratic distance between $\tau_X(X)$ and $1$. The HSI inequality
obtained in Theorem \ref{HSI} realizes this intuition by providing an
explicit bound of the relative entropy of $X$ with respect to
$\gamma_{\alpha,\lambda}$ in terms of the Stein discrepancy and the
standardized Fisher information of $X$. It improves upon the
logarithmic Sobolev inequality for the gamma case as obtained in
Proposition \ref{LSI}. Moreover, the appearance of the differential
operator $\partial_x^{\sigma}$ is canonically linked to the Laguerre
dynamic. Indeed, under the action of this operator, the Laguerre
semigroup admits stochastic representations such as the intertwining
relation of Lemma \ref{MehInter} and the Bismut-type formulae of Lemma
\ref{Bismut}. These representations are pivotal to establish a
fundamental representation of the standardized Fisher information
structure along the smart path (see Proposition \ref{FisherRep}).

\subsection{Structure of the paper}
\label{sec:structure-paper}

The article is organized as follows. In the next section, we define
the stochastic representation of the interpolation scheme along the
Laguerre dynamic and derive several identities in law as well as
standard properties. Section \ref{sec-Fish} is devoted to the study of
the Fisher information structure along the gamma smart path. In
Section \ref{sec-DeBruijn}, we prove the local and integrated version
of De Bruijn formulae for the gamma case for $\alpha\geq 1/2$. Section \ref{sec-LSI}
contains the new proof of the logarithmic Sobolev inequality for the
gamma case with $\alpha\geq1/2$ and Section \ref{sec-HSI} contains the
tools in order to establish the new HSI inequality for the gamma law
with $\alpha\geq1/2$. Finally, Section \ref{sec-Append} collects the
more technical proofs regarding analytical properties of the density
of $X_{\tau}$.

\section{Gamma Interpolation}\label{sec-GamInter}


In this section, we consider a probability distribution on
$]0,+\infty[$ denoted by $\mathbb{P}_X$. We denote by $X$ the
associated random variable. In the next definition, we introduce the
parametrized random variable, namely the gamma smart path, which
``interpolates'' between the gamma law and the random variable
$X$. This random variable $X_{\tau}$, with $\tau\in(0,1)$, is built
via a three-stage explicit procedure depending on $X$.
\begin{defi}[The gamma smart path]\label{main}
  Let $\tau\in (0,1)$. Let $x\in(0,+\infty)$ be drawn according to the
  law of $X$. Let $K(\tau,x,\lambda)$ be a Poisson random variable of
  parameter $x\lambda\tau/(1-\tau)$ independent of $X$. Let
  $Y(\tau,x,\lambda)$ be a random variable which is drawn in the
  following way:
\begin{align*}
Y(\tau,x,\lambda)=
\begin{cases}
0 & K(\tau,x,\lambda)=0\\
\tilde{\gamma}\big(k,\frac{\lambda \tau}{1-\tau}\big) & K(\tau,x,\lambda)=k,
\end{cases}
\end{align*}
where $\tilde{\gamma}\big(k,\frac{\lambda\tau}{1-\tau}\big)$ is a gamma random variable independent of $\{X,K(\tau,x,\lambda)\}$. Then, we define $X_{\tau}$ by:
\begin{align*}
X_{\tau}=\big(1-\tau\big)\gamma(\alpha,\lambda)+\tau Y(\tau,X,\lambda),
\end{align*}
where $\gamma(\alpha,\lambda)$ is a gamma random variable independent of $\{X,K(\tau,x,\lambda), \tilde{\gamma}\big(K(\tau,x,\lambda),\frac{\lambda\tau}{1-\tau}\big)\}$.
Moreover, we denote the density of $X_{\tau}$ by $g^{(\alpha,\lambda)}(\tau,.)$ which is completely characterized by the following formulae:
\begin{align}
\forall u>0,\ g^{(\alpha,\lambda)}(\tau,u)&=\dfrac{\lambda}{1-\tau}\bigg(\dfrac{u}{\tau}\bigg)^{\frac{\alpha-1}{2}}\exp\bigg(-\dfrac{\lambda u}{1-\tau}\bigg)\nonumber\\
&\times\int_0^{+\infty}(\frac{1}{x})^{\frac{\alpha-1}{2}}\exp\bigg(-\dfrac{\lambda\tau x}{1-\tau}\bigg)I_{\alpha-1}\bigg(\dfrac{2\lambda\sqrt{ux\tau}}{1-\tau}\bigg)\mathbb{P}_X(dx),\label{den}\\
\forall \mu>0,\ L(g^{(\alpha,\lambda)}(\tau,.))(\mu)&=\dfrac{1}{\bigg(1+\frac{\mu}{\lambda}(1-\tau)\bigg)^{\alpha}}\int_0^{+\infty}\exp\bigg(-\dfrac{\mu \tau x}{1+\frac{\mu}{\lambda}(1-\tau)}\bigg)\mathbb{P}_X(dx),\label{lap}
\end{align}
with $L$ being the Laplace transform operator. 
\end{defi}
\begin{rem}
From the definition of the random variable $X_{\tau}$, it is easy to compute its Laplace transform providing formula (\ref{lap}). Moreover, using the series expansion of the modified Bessel function of the first kind of order $\alpha-1$, it is easy to compute the Laplace transform of the right-hand side of formula (\ref{den}) leading to the corresponding formula for the density of $X_{\tau}$. The details are left to the interested reader.
\end{rem}
\noindent
In the next simple lemma, we prove that the law of the random
variable $X_{\tau}$ interpolates between a gamma law of parameters
$(\alpha,\lambda)$ and the probability measure $\mathbb{P}_X$.
\begin{lem}\label{Inter}
We have:
\begin{align*}
&\underset{\tau\rightarrow0^+}{\lim}X_{\tau}\overset{\mathcal{L}}{=}\gamma(\alpha,\lambda),\\
&\underset{\tau\rightarrow1^-}{\lim}X_{\tau}\overset{\mathcal{L}}{=}X.
\end{align*}
\end{lem}
\begin{proof}
Let $\mu>0$. The Laplace transform of $g^{(\alpha,\lambda)}(\tau,.)$ is given by (\ref{lap}):
\begin{align*}
\forall \mu>0,\ L(g^{(\alpha,\lambda)}(\tau,.))(\mu)&=\dfrac{1}{\bigg(1+\frac{\mu}{\lambda}(1-\tau)\bigg)^{\alpha}}\int_0^{+\infty}\exp\bigg(-\dfrac{\mu \tau x}{1+\frac{\mu}{\lambda}(1-\tau)}\bigg)\mathbb{P}_X(dx).
\end{align*}
Note that, for every $x\ne0$:
\begin{align*}
&\underset{\tau\rightarrow0^+}{\lim}\dfrac{1}{\bigg(1+\frac{\mu}{\lambda}(1-\tau)\bigg)^{\alpha}}\exp\bigg(-\dfrac{\mu \tau x}{1+\frac{\mu}{\lambda}(1-\tau)}\bigg)=\dfrac{1}{\big(1+\frac{\mu}{\lambda}\big)^{\alpha}}=L(\gamma(\alpha,\lambda))(\mu),\\
&\underset{\tau\rightarrow1^-}{\lim}\dfrac{1}{\bigg(1+\frac{\mu}{\lambda}(1-\tau)\bigg)^{\alpha}}\exp\bigg(-\dfrac{\mu \tau x}{1+\frac{\mu}{\lambda}(1-\tau)}\bigg)=\exp(-\mu x),\\
&|\dfrac{1}{\bigg(1+\frac{\mu}{\lambda}(1-\tau)\bigg)^{\alpha}}\exp\bigg(-\dfrac{\mu \tau x}{1+\frac{\mu}{\lambda}(1-\tau)}\bigg)|\leq 1.
\end{align*}
Thus, by Lebesgue dominated convergence theorem, we obtain:
\begin{align*}
&\underset{\tau\rightarrow0^+}{\lim}L(g^{(\alpha,\lambda)}(\tau,.))(\mu)=L(\gamma(\alpha,\lambda))(\mu),\\
&\underset{\tau\rightarrow1^-}{\lim}L(g^{(\alpha,\lambda)}(\tau,.))(\mu)=L(\mathbb{P}_X)(\mu).
\end{align*}
We conclude by a Laplace transform version of Levy Theorem for probability measure on $(0,+\infty)$.
\end{proof}
\noindent
As a direct application of Definition \ref{main}, we obtain the following formula for the mean of the smart path $X_{\tau}$.
\begin{cor}\label{mean}
For any $\tau>0$, we have:
\begin{align*}
\mathbb{E}\big[X_{\tau}\big]=(1-\tau)\dfrac{\alpha}{\lambda}+\tau\mathbb{E}\big[X\big].
\end{align*}
\end{cor}
\noindent
We note as well the following nice property regarding convolutions.
\begin{cor}\label{ConvSP}
Let $\{X_i,\ i\in1,...,N\}$ be a collection of independent random variables almost surely positive. Let $X=\sum_{i=1}^{N}X_i$. Then, we have:
\begin{align*}
X_{\tau}\big(X\big)&\overset{\mathcal{L}}{=}(1-\tau)\gamma(\alpha,\lambda)+\tau\sum_{i=1}^{N}Y^i\big(\tau,X_i,\lambda\big),\\
&\overset{\mathcal{L}}{=}\sum_{i=1}^N\bigg((1-\tau)\gamma^i\big(\frac{\alpha}{N},\lambda\big)+\tau Y^i\big(\tau,X_i,\lambda\big)\bigg).
\end{align*}
\end{cor}
\begin{proof} 
This is a direct application of the Laplace transform formula for $X_{\tau}$, property of Laplace transform on convolutions and the definition of $Y^i(\tau,X_i,\lambda)$. 
\end{proof}
\noindent
In the particular  cases when $\alpha=p/2$, we obtain another representation in law for the smart path, $X_{\tau}$ which could be of interest. This is linked with the classical fact that the squared radial Ornstein-Uhlenbeck process (or Laguerre process) of parameters $(p/2,\lambda)$ can be represented as the squared of the euclidean norm of a $p$-dimensional Ornstein-Uhlenbeck process with parameter $\lambda$. Indeed, when the law of the random variable $X$ admits a density, $f_X$, the density of $X_{\tau}$ is exactly the Lebesgue adjoint of the Laguerre semigroup acting on $f_X$ after the time change $t=-\log(\tau)/\lambda$.
\begin{prop}\label{Melherlike2}
Let $\alpha=p/2$ with $p\in\mathbb{N}^*$ and $\lambda>0$. Let $(Z_1,...,Z_p)$ be a Gaussian random vector with mean zero and the identity matrix as its covariance matrix. Then, we have:
\begin{align*}
X_{\tau}\overset{\mathcal{L}}{=}\sum_{i=1}^p\bigg(\sqrt{\tau}\sqrt{\frac{X}{p}}+\sqrt{\frac{1-\tau}{2\lambda}}Z_i\bigg)^2.
\end{align*} 
\end{prop}
\begin{proof}
Let us prove that the right hand side of the previous equality as the same Laplace transform as $X_{\tau}$. By definition, we have:
\begin{align*}
&\forall\mu>0,\
  \mathbb{E}\bigg[\exp\bigg(-\mu\sum_{i=1}^p\bigg(\sqrt{\tau}\sqrt{\frac{X}{p}}+\sqrt{\frac{1-\tau}{2\lambda}}Z_i\bigg)^2\bigg)\bigg]\\
&=\int_0^{\infty}\mathbb{E}\bigg[\exp\bigg(-\mu\sum_{i=1}^p\bigg(\sqrt{\tau}\sqrt{\frac{x}{p}}+\sqrt{\frac{1-\tau}{2\lambda}}Z_i\bigg)^2\bigg)\bigg]
 \mathbb{P}_X(dx),\\
&=\int_0^{\infty}\mathbb{E}\bigg[\exp\bigg(-\mu\frac{(1-\tau)}{2\lambda}\sum_{i=1}^p\dfrac{\bigg(\sqrt{\tau}\sqrt{\frac{x}{p}}+\sqrt{\frac{1-\tau}{2\lambda}}Z_i\bigg)^2}{\frac{1-\tau}{2\lambda}}\bigg)\bigg] \mathbb{P}_X(dx).
\end{align*}
Conditionally to $X$, we recognize the Laplace transform of a
non-central chi-squared random variable with parameters
$(p,2\lambda\tau x/(1-\tau))$ evaluated at $\mu
(1-\tau)/2\lambda$. Thus, we obtain:
\begin{align*}
\mathbb{E}\bigg[\exp\bigg(-\mu\sum_{i=1}^p\bigg(\sqrt{\tau}\sqrt{\frac{X}{p}}+\sqrt{\frac{1-\tau}{2\lambda}}Z_i\bigg)^2\bigg)\bigg]&=\int_0^{\infty}\bigg(\dfrac{1}{1+\frac{\mu(1-\tau)}{\lambda}}\bigg)^{\frac{p}{2}}\exp\bigg(\dfrac{-\mu\tau x}{1+\frac{\mu(1-\tau)}{\lambda}}\bigg)\mathbb{P}_X(dx),\\
&=\bigg(\dfrac{1}{1+\frac{\mu(1-\tau)}{\lambda}}\bigg)^{\frac{p}{2}}\int_0^{+\infty}\exp\bigg(\dfrac{-\mu\tau x}{1+\frac{\mu(1-\tau)}{\lambda}}\bigg)\mathbb{P}_X(dx),\\
&=L(X_{\tau})(\mu).\quad
\end{align*}
\end{proof}
\noindent
The next result is a combination of the two previous representations which is available for $\alpha>1/2$ et which is of central importance for Sections \ref{sec-LSI} and \ref{sec-HSI}.
\begin{cor}\label{Melherlike3}
Let $\alpha>1/2$ and $\lambda>0$. Let $Z$ be a standard normal random variable and $\gamma(\alpha-1/2,\lambda)$ a gamma random variable of parameters $(\alpha-1/2,\lambda)$ such that $\big(Z,X,\gamma(\alpha-1/2,\lambda)\big)$ are independent. Then,
\begin{align*}
X_{\tau}\overset{\mathcal{L}}{=}(1-\tau)\gamma\big(\alpha-\frac{1}{2},\lambda\big)+\bigg(\sqrt{\tau}\sqrt{X}+\sqrt{\dfrac{1-\tau}{2\lambda}}Z\bigg)^2.
\end{align*}
\end{cor}
\begin{proof}
The Laplace transform of $X_{\tau}$ is equal to:
\begin{align*}
\forall\mu>0, L(X_{\tau})(\mu)&=\dfrac{1}{\big(1+\frac{\mu}{\lambda}(1-\tau)\big)^{\alpha}}L\big(X\big)\big(\dfrac{\mu \tau}{1+\frac{\mu}{\lambda}(1-\tau)}\big),\\
&=\dfrac{1}{\big(1+\frac{\mu}{\lambda}(1-\tau)\big)^{\alpha-\frac{1}{2}}}\dfrac{1}{\big(1+\frac{\mu}{\lambda}(1-\tau)\big)^{\frac{1}{2}}}L\big(X\big)\big(\dfrac{\mu \tau}{1+\frac{\mu}{\lambda}(1-\tau)}\big),\\
&=L\bigg((1-\tau)\gamma\big(\alpha-\frac{1}{2},\lambda\big)\bigg)(\mu)L(X'_{\tau})(\mu),
\end{align*}
where $X'_{\tau}$ is a smart path of parameters $(1/2,\lambda)$ independent of $\gamma\big(\alpha-1/2,\lambda\big)$. We apply Proposition \ref{Melherlike2} to $X'_{\tau}$ to conclude the proof.
\end{proof}
\noindent
We end this section by two lemmata which provide explicit formulae for
the $\beta$-moments and the exponential moments of the smart path
$X_{\tau}$ under the assumptions that such moments exist for the
initial law, $\mathbb{P}_{X}$.
\begin{lem}\label{Moments}
Let $\alpha>0$, $\lambda>0$ and $\beta$ such that $\alpha+\beta>0$. Let $X$ be a strictly positive random variable such that $\mathbb{E}\big[X^{\beta}\big]<+\infty$. Then, $\mathbb{E}\big[(X_{\tau})^{\beta}\big]<+\infty$ and we have:
\begin{align*}
\mathbb{E}\big[(X_{\tau})^{\beta}\big]=\bigg(\dfrac{\lambda}{1-\tau}\bigg)^{-\beta}\dfrac{\Gamma(\alpha+\beta)}{\Gamma(\alpha)}\int_0^{+\infty}\exp\bigg(-\dfrac{\lambda x \tau}{1-\tau}\bigg)\tensor[_1]{F}{_1}\bigg(\alpha+\beta,\alpha,\dfrac{\lambda \tau x}{1-\tau}\bigg)\mathbb{P}_X(dx).
\end{align*}
where $\tensor[_1]{F}{_1}$ is the Kummer confluent hypergeometric function (of the first kind).  
\end{lem}
\begin{proof} 
By definition, we have:
\begin{align*}
\mathbb{E}\big[(X_{\tau})^{\beta}\big]&=\int_0^{+\infty}u^{\beta}g^{(\alpha,\lambda)}(\tau,u)du,\\
&=\dfrac{\lambda}{1-\tau}\big(\frac{1}{\tau}\big)^{\frac{\alpha-1}{2}}\int_0^{+\infty}\big(\frac{1}{x}\big)^{\frac{\alpha-1}{2}}\exp\bigg(-\dfrac{\lambda\tau x}{1-\tau}\bigg)\\
&\times\bigg(\int_0^{+\infty}u^{\beta+\frac{\alpha-1}{2}}\exp\bigg(-\dfrac{u\lambda}{1-\tau}\bigg)I_{\alpha-1}\bigg(\dfrac{2\lambda\sqrt{u\tau x}}{1-\tau}\bigg)du\bigg)\mathbb{P}_X(dx).
\end{align*}
Expanding the modified Bessel functions of the first kind into power
series, we have the following (since $\alpha+\beta>0$):  
\begin{align*}
&\int_0^{+\infty}u^{\beta+\frac{\alpha-1}{2}}\exp\bigg(-\dfrac{u\lambda}{1-\tau}\bigg)I_{\alpha-1}\bigg(\dfrac{2\lambda\sqrt{u\tau
  x}}{1-\tau}\bigg)du\\
&=\sum_{n=0}^{+\infty}\frac{1}{n!}\dfrac{1}{\Gamma(\alpha+n)}\bigg(\dfrac{\lambda\sqrt{x\tau}}{1-\tau}\bigg)^{\alpha-1+2n}\int_0^{+\infty}u^{\beta+\alpha-1+n}\exp\bigg(-\dfrac{u\lambda}{1-\tau}\bigg)du,\\
&=\sum_{n=0}^{+\infty}\frac{1}{n!}\dfrac{1}{\Gamma(\alpha+n)}\bigg(\dfrac{\lambda\sqrt{x\tau}}{1-\tau}\bigg)^{\alpha-1+2n}\bigg(\dfrac{1-\tau}{\lambda}\bigg)^{\alpha+\beta}\Gamma(\alpha+\beta+n),\\
&=\bigg(\dfrac{\lambda}{1-\tau}\bigg)^{-(\beta+1)}\big(x\tau\big)^{\frac{\alpha-1}{2}}\dfrac{\Gamma(\alpha+\beta)}{\Gamma(\alpha)}\tensor[_1]{F}{_1}\bigg(\alpha+\beta,\alpha,\dfrac{\lambda \tau x}{1-\tau}\bigg).
\end{align*}
Thus,
\begin{align*}
\mathbb{E}\big[(X_{\tau})^{\beta}\big]=\bigg(\dfrac{\lambda}{1-\tau}\bigg)^{-\beta}\dfrac{\Gamma(\alpha+\beta)}{\Gamma(\alpha)}\int_0^{+\infty}\exp\bigg(-\dfrac{\lambda x \tau}{1-\tau}\bigg)\tensor[_1]{F}{_1}\bigg(\alpha+\beta,\alpha,\dfrac{\lambda \tau x}{1-\tau}\bigg)\mathbb{P}_X(dx).
\end{align*}
To conclude, we need to study the finiteness of the previous integral and in particular the integrability of the integrand at $+\infty$. But, we have the following asymptotic:
\begin{align*}
\tensor[_1]{F}{_1}(\alpha+\beta,\alpha,z)\underset{z\rightarrow+\infty}{\sim}\dfrac{\Gamma(\alpha)}{\Gamma(\alpha+\beta)}\exp(z)z^{\beta}.
\end{align*}
Moreover, by assumptions, $\mathbb{E}\big[(X)^{\beta}\big]<+\infty$. This concludes the proof of the lemma.
\end{proof} 
\begin{lem}\label{ExpoMoments}
Let $\mu>0$ such that $\mathbb{E}\big[\exp\big(\mu X\big)\big]<+\infty$. Then, there exists $s(\mu,\tau)>0$, such that:
\begin{align*}
\forall\tau\in (0,1),\ \mathbb{E}\big[e^{s(\mu,\tau)X_{\tau}}\big]<+\infty.
\end{align*}
Moroever, we have:
\begin{align*}
&\forall\tau\in (0,1),\ \mathbb{E}\big[e^{s(\mu,\tau)X_{\tau}}\big]=\dfrac{1}{\bigg(1+\dfrac{\mu(1-\tau)}{\lambda\tau}\bigg)^{\alpha}}\mathbb{E}[e^{\mu X}],\\
&s(\mu,\tau)=\dfrac{\mu}{\tau+\frac{\mu(1-\tau)}{\lambda}}.
\end{align*}
\end{lem}
\begin{proof}
Let $s(\mu,\tau)$ be as in the lemma. By definition, we have:
\begin{align*}
\forall\tau\in (0,1),\ \mathbb{E}\big[e^{s(\mu,\tau)X_{\tau}}\big]&=\bigg(\dfrac{\lambda}{1-\tau}\bigg)\big(\frac{1}{\tau}\big)^{\frac{\alpha-1}{2}}\int_0^{+\infty}\big(\frac{1}{x}\big)^{\frac{\alpha-1}{2}}\exp\bigg(-\dfrac{\lambda\tau x}{1-\tau}\bigg)\\
&\times\bigg(\int_0^{+\infty}u^{\frac{\alpha-1}{2}}\exp\bigg(-u\big(\frac{\lambda}{1-\tau}-s(\mu,\tau)\big)\bigg)I_{\alpha-1}\bigg(\dfrac{2\lambda\sqrt{ux\tau}}{1-\tau}\bigg)du\bigg)\mathbb{P}_X(dx).
\end{align*}
Expanding the modified Bessel function of the first kind of order $\alpha-1$ into power series, we obtain:
\begin{align*}
\int_0^{+\infty}u^{\frac{\alpha-1}{2}}\exp\bigg(-u\big(\frac{\lambda}{1-\tau}-s(\mu,\tau)\big)\bigg)I_{\alpha-1}\bigg(\dfrac{2\lambda\sqrt{ux\tau}}{1-\tau}\bigg)du&=\sum_{n=0}^{+\infty}\frac{1}{n!}\frac{1}{\Gamma(\alpha+n)}\bigg(\dfrac{\lambda\sqrt{\tau x}}{1-\tau}\bigg)^{\alpha-1+2n}\\
&\times\int_0^{+\infty}u^{\alpha-1+n}\exp\bigg(-u\big(\frac{\lambda}{1-\tau}-s(\mu,\tau)\big)\bigg)du.
\end{align*}
Note that the integral on the right hand side is finite since $s(\mu,\tau)<\lambda/(1-\tau)$. Thus,
\begin{align*}
\int_0^{+\infty}u^{\frac{\alpha-1}{2}}\exp\bigg(-u\big(\frac{\lambda}{1-\tau}-s(\mu,\tau)\big)\bigg)I_{\alpha-1}\bigg(\dfrac{2\lambda\sqrt{ux\tau}}{1-\tau}\bigg)du&=\bigg(\dfrac{\lambda\sqrt{x\tau}}{1-\tau}\bigg)^{\alpha-1}\dfrac{1}{\big(\frac{\lambda}{1-\tau}-s(\mu,\tau)\big)^{\alpha}}\\
&\times\exp\bigg(\dfrac{\lambda\tau x}{1-\tau}\dfrac{1}{1-\frac{s(\mu,\tau)(1-\tau)}{\lambda}}\bigg).
\end{align*}
Consequently, we have:
\begin{align*}
\mathbb{E}\big[e^{s(\mu,\tau)X_{\tau}}\big]=\dfrac{1}{\bigg(1-\frac{s(\mu,\tau)(1-\tau)}{\lambda}\bigg)^{\alpha}}\int_0^{+\infty}\exp\bigg(\dfrac{s(\mu,\tau)x\tau}{1-\frac{s(\mu,\tau)(1-\tau)}{\lambda}}\bigg)\mathbb{P}_X(dx).
\end{align*}
Using the formula for $s(\mu,\tau)$, we obtain the desired result.
\end{proof}


\section{Fisher Information structure along the gamma smart path}\label{sec-Fish}

In this subsection, we introduce a localized version of Fisher information which is relevant in order to establish the De Bruijn formula for the gamma law. Indeed, this quantity appears naturally at a local level along the smart path, $X_{\tau}$, when computing the derivative of the relative entropy $D(X_{\tau}\|\gamma_{\alpha,\lambda})$ with respect to $\tau$.
\begin{defi}\label{NewFisher}
Let $\tau\in (0,1)$ and $X$ be a positive random variable with density $f_X$. We define $I^{\tau}_{\gamma}(X)$ by:
\begin{align*}
I^{\tau}_{\gamma}(X)=\mathbb{E}\big[X\big(\rho_X(X)-\rho_{\alpha,\frac{\lambda}{(1-\tau)}}(X)\big)^2\big],
\end{align*}
where $\rho_X(u)=\partial_u\big(\log(f_X(u))\big)$ and $\rho_{\alpha,\lambda/(1-\tau)}(u)=(\alpha-1)/u-\lambda/(1-\tau)$.
\end{defi} 
\begin{rem}
\begin{itemize}
\item Note that if $X$ has a gamma law with parameters $(\alpha,\lambda)$, we have:
\begin{align*}
I^{\tau}_{\gamma}(X)&=\mathbb{E}\big[X\big(\dfrac{\alpha-1}{X}-\lambda-\dfrac{\alpha-1}{X}+\dfrac{\lambda}{1-\tau}\big)^2\big],\\
&=\dfrac{\lambda^2\tau^2}{(1-\tau)^2}\mathbb{E}\big[X\big],\\
&=\lambda\alpha\dfrac{\tau^2}{(1-\tau)^2}.
\end{align*}
\item Note that the previous quantity is actually well-defined for any
values of $\alpha$ whereas the quantities
$\mathbb{E}[X\big(\rho_X(X)\big)^2]$ and
$\mathbb{E}[\big(\rho_X(X)\big)^2]$ are finite only for gamma laws with
shape parameters $\alpha>1$ and $\alpha>2$ respectively.
\end{itemize}
\end{rem}
\noindent
Thus, we introduce the following standardized localized Fisher information with respect to the gamma law of parameters $(\alpha,\lambda)$:
\begin{align*}
J_{st,\gamma}(X)=\mathbb{E}\big[X\big(\rho_X(X)-\rho_{\alpha,\frac{\lambda}{(1-\tau)}}(X)-\dfrac{\lambda\tau}{1-\tau}\big)^2\big].
\end{align*}
We note that this standardized localized Fisher information is actually equal to:
\begin{align*}
J_{st,\gamma}(X)=\mathbb{E}\big[X\big(\rho_X(X)-\rho_{\alpha,\lambda}(X)\big)^2\big].
\end{align*}
Regarding these quantities along the smart path, we have the following results.
\begin{prop}\label{FiniteFisher}
Assume that $X$ admits a first moment. Then, we have,
\begin{align*}
\forall\tau\in(0,1),\ I^{\tau}_{\gamma}(X_{\tau})<+\infty
\end{align*}
Moreover, if $\mathbb{E}\big[X\big]=\alpha/\lambda$, we have:
\begin{align*}
&\forall\alpha\geq 1,\ I^{\tau}_{\gamma}(X_{\tau})\leq \dfrac{\alpha\lambda\tau}{(1-\tau)^2},\\
&\forall\alpha\in(0,1),\ I^{\tau}_{\gamma}(X_{\tau})\leq \bigg(1+\dfrac{1}{\alpha}\bigg)\dfrac{\alpha\lambda\tau}{(1-\tau)^2}.
\end{align*} 
\end{prop}
\begin{proof}
By definition of $I^{\tau}_{\gamma}(.)$, we have:
\begin{align*}
I^{\tau}_{\gamma}(X_{\tau})&=\mathbb{E}\big[X_{\tau}\big(\rho_{X_{\tau}}(X_{\tau})-\rho_{\alpha,\frac{\lambda}{(1-\tau)}}(X_{\tau})\big)^2\big],\\
&=\mathbb{E}\big[\dfrac{1}{X_{\tau}}\big(\overline{\rho}_{X_{\tau}}(X_{\tau})-\overline{\rho}_{\alpha,\frac{\lambda}{(1-\tau)}}(X_{\tau})\big)^2\big],
\end{align*}
where $\overline{\rho}_X(u)=\partial_u\big(uf_X(u)\big)/f_X(u)$ and $\overline{\rho}_{\alpha,\frac{\lambda}{(1-\tau)}}(u)=\alpha-\lambda u/(1-\tau)$. Note that:
\begin{align*}
&\overline{\rho}_X(u)=u\rho_X(u)+1,\\
&\forall\phi\in C^{\infty}_c\big((0,+\infty)\big),\ \mathbb{E}\big[\overline{\rho}_X(X)\phi(X)\big]=-\mathbb{E}\big[X\phi^{(1)}(X)\big],\\
&\forall\phi\in C^{\infty}_c\big((0,+\infty)\big),\ \mathbb{E}\big[\big(\alpha-\lambda \gamma(\alpha,\lambda)\big)\phi(\gamma(\alpha,\lambda))\big]=-\mathbb{E}\big[\gamma(\alpha,\lambda)\phi^{(1)}(\gamma(\alpha,\lambda))\big].
\end{align*}
Let $\phi\in C^{\infty}_c\big((0,+\infty)\big)$. We have:
\begin{align*}
\mathbb{E}\big[\overline{\rho}_{X_{\tau}}(X_{\tau})\phi\big(X_{\tau}\big)\big]&=-\mathbb{E}\big[X_{\tau}\phi^{(1)}(X_{\tau})\big],\\
&=-(1-\tau)\mathbb{E}\big[\gamma(\alpha,\lambda)\phi^{(1)}(X_{\tau})\big]-\tau\mathbb{E}\big[Y(\tau,X,\lambda)\phi^{(1)}(X_{\tau})\big],\\
&=\mathbb{E}\big[(\alpha-\lambda\gamma(\alpha,\lambda))\phi\big(X_{\tau}\big)\big]-\tau\mathbb{E}\big[Y(\tau,X,\lambda)\phi^{(1)}(X_{\tau})\big],
\end{align*}
where we have used Definition \ref{main} and Stein formula for the Gamma distribution. Let us deal with the second term in details:
\begin{align*}
& \mathbb{E}\big[Y(\tau,X,\lambda)\phi^{(1)}((1-\tau)\gamma(\alpha,\lambda)+\tau
  Y(\tau,X,\lambda))\big]\\
&=\int_0^{+\infty}\mathbb{E}\big[Y(\tau,x,\lambda)\phi^{(1)}((1-\tau)\gamma(\alpha,\lambda)+\tau Y(\tau,x,\lambda))\big]\mathbb{P}_X(dx),\\
&=\sum_{k=1}^{+\infty}\int_0^{+\infty}\mathbb{E}\big[\gamma(k,\frac{\lambda \tau}{1-\tau})\phi^{(1)}((1-\tau)\gamma(\alpha,\lambda)+\tau \gamma(k,\frac{\lambda \tau}{1-\tau})\big]\\
&\times\mathbb{P}\big(K(\tau,x,\lambda)=k\big)\mathbb{P}_X(dx),\\
&=\frac{1}{\tau}\sum_{k=1}^{+\infty}\int_0^{+\infty}\mathbb{E}\big[(\frac{\lambda \tau}{1-\tau}\gamma(k,\frac{\lambda \tau}{1-\tau})-k)\phi((1-\tau)\gamma(\alpha,\lambda)\\
&+\tau \gamma(k,\frac{\lambda \tau}{1-\tau}))\big]\times\mathbb{P}\big(K(\tau,x,\lambda)=k\big)\mathbb{P}_X(dx),\\
&=\frac{1}{\tau}\mathbb{E}\big[(\frac{\lambda \tau}{1-\tau}Y(\tau,X,\lambda)-K(\tau,X,\lambda))\phi(X_{\tau})\big].
\end{align*}
Thus, we obtain:
\begin{align*}
\mathbb{E}\big[\overline{\rho}_{X_{\tau}}(X_{\tau})\phi\big(X_{\tau}\big)\big]&=\mathbb{E}\big[(\alpha-\lambda\gamma(\alpha,\lambda)+K(\tau,X,\lambda)-\frac{\lambda \tau}{1-\tau}Y(\tau,X,\lambda))\phi\big(X_{\tau}\big)\big],\\
&=\mathbb{E}\big[(\alpha+K(\tau,X,\lambda)-\frac{\lambda}{1-\tau}X_{\tau})\phi\big(X_{\tau}\big)\big].
\end{align*}
Therefore, we have:
\begin{align*}
I^{\tau}_{\gamma}(X_{\tau})&=\mathbb{E}\big[\frac{1}{X_{\tau}}\big(\mathbb{E}\big[K(\tau,X,\lambda)|X_{\tau}\big]\big)^2\big],\\
&\leq \mathbb{E}\big[\dfrac{K(\tau,X,\lambda)^2}{X_{\tau}}\big],\\
&\leq \sum_{k=1}^{+\infty}\int_0^{+\infty}\mathbb{P}\big(K(\tau,x,\lambda)=k\big)\mathbb{E}\big[\dfrac{k^2}{\gamma(\alpha+k,\frac{\lambda}{1-\tau})}\big]\mathbb{P}_X(dx),\\
&\leq \frac{\lambda}{1-\tau}\sum_{k=1}^{+\infty}\int_0^{+\infty}\bigg(\dfrac{\lambda x\tau}{1-\tau}\bigg)^k\frac{1}{k!}\exp\bigg(-\dfrac{\lambda x\tau}{1-\tau}\bigg)\frac{k^2}{\alpha+k-1}\mathbb{P}_X(dx),\\
&\leq \frac{\lambda^2\tau}{(1-\tau)^2}\int_0^{+\infty}\bigg(\sum_{k=0}^{+\infty}\bigg(\dfrac{\lambda x\tau}{1-\tau}\bigg)^k\frac{1}{k!}\dfrac{k+1}{k+\alpha}\bigg)\exp\bigg(-\dfrac{\lambda x\tau}{1-\tau}\bigg)x\mathbb{P}_X(dx).
\end{align*}
If $\alpha\geq 1$, we have the following bound:
\begin{align*}
I^{\tau}_{\gamma}(X_{\tau})\leq \frac{\lambda^2\tau}{(1-\tau)^2}\mathbb{E}\big[X\big]<+\infty.
\end{align*}
If $\alpha\in (0,1)$, we have:
\begin{align*}
I^{\tau}_{\gamma}(X_{\tau})\leq \frac{\lambda^2\tau}{(1-\tau)^2}\bigg(1+\frac{1}{\alpha}\bigg)\mathbb{E}\big[X\big]<+\infty
\end{align*}
This ends the proof of the proposition.
\end{proof}
\begin{prop}\label{CramerRao}
Assume that $X$ admits a first moment and that $\mathbb{E}\big[X\big]=\alpha/\lambda$. Then, we have:
\begin{align*}
J_{st,\gamma}(X_{\tau})=I^{\tau}_{\gamma}(X_{\tau})-\dfrac{\alpha\lambda\tau^2}{(1-\tau)^2}\geq 0.
\end{align*}
Moreover, when $\alpha\geq 1$, we have the following upper bound:
\begin{align*}
J_{st,\gamma}(X_{\tau})\leq \dfrac{\alpha\lambda\tau}{1-\tau}.
\end{align*}
\end{prop}
\begin{proof}
By definition of the standardized localized Fisher information, we have:
\begin{align*}
J_{st,\gamma}(X_{\tau})&=\mathbb{E}\big[X_{\tau}\big(\rho_{X_{\tau}}(X_{\tau})-\rho_{\alpha,\frac{\lambda}{(1-\tau)}}(X_{\tau})-\dfrac{\lambda\tau}{1-\tau}\big)^2\big],\\
&=I^{\tau}_{\gamma}(X_{\tau})-2\dfrac{\lambda\tau}{1-\tau}\mathbb{E}\big[X_{\tau}\big(\rho_{X_{\tau}}(X_{\tau})-\rho_{\alpha,\frac{\lambda}{(1-\tau)}}(X_{\tau})\big)\big]+\mathbb{E}\big[X_{\tau}\big]\dfrac{\lambda^2\tau^2}{(1-\tau)^2},\\
&=I^{\tau}_{\gamma}(X_{\tau})-2\dfrac{\lambda\tau}{1-\tau}\bigg(-\alpha+\dfrac{\lambda}{1-\tau}\mathbb{E}\big[X_{\tau}\big]\bigg)+\dfrac{\alpha\lambda\tau^2}{(1-\tau)^2},\\
&=I^{\tau}_{\gamma}(X_{\tau})-\dfrac{\alpha\lambda\tau^2}{(1-\tau)^2}.
\end{align*}
where we have used Corollary \ref{mean} as well as a classical property of the score function.
\end{proof}
\begin{rem}\label{CramerRao2}:
The previous result actually holds for any positive random variable with finite first moment equal to $\alpha/\lambda$ and with finite localized Fisher information. Namely, we have:
\begin{align*}
J_{st,\gamma}(X)=I^{\tau}_{\gamma}(X)-\dfrac{\alpha\lambda\tau^2}{(1-\tau)^2}.
\end{align*}
\end{rem}


\section{A new formulation of De Bruijn identity}\label{sec-DeBruijn}

In this section, we assume that $\alpha\geq 1/2$.
By definition, for every $\tau\in (a,b)$, we have:
\begin{align*}
D(X_{\tau}\|\gamma_{\alpha,\lambda})=\int_0^{+\infty} 
g^{(\alpha,\lambda)}(\tau,u)\log\bigg(
\dfrac{g^{(\alpha,\lambda)}(\tau,u)}{\gamma_{\alpha,\lambda}(u)}\bigg)du. 
\end{align*}
Note that $D(X_{\tau}\|\gamma_{\alpha,\lambda})$ is finite by Lemma
\ref{tech2}, Lemma \ref{tech3} and together with the assumption that
$X$ has finite $\alpha+4$ moment (which ensures that $X_{\tau}$ has
finite $\alpha+4$ moment by Lemma \ref{Moments}). In order to obtain a
De Bruijn formula we need to be able to interchange derivatives and
integrals in the above and thus bound the integrand uniformly in
$\tau$. In doing so we will identify a moment condition on $X$ which
is the only assumption that will be necessary for our formula to hold.
\begin{theo}\label{DeBruijn2}
Let $X$ be an almost surely positive random variable with finite $\alpha+4$ moment. Then, we have:
\begin{align*}
\dfrac{d}{d\tau}\bigg(D(X_{\tau}\|\gamma_{\alpha,\lambda})\bigg)=\frac{1}{\lambda\tau}J_{st,\gamma}(X_{\tau}).
\end{align*}
\end{theo}
\begin{rem}\label{better}
  As already discussed in the Introduction, this local version of De
  Bruijn formula should be compared with Proposition $5.2.2$ in
  \cite{BGL} where a De Bruijn formula for the gamma law has been
  obtained by semigroup arguments. Although formally equivalent, our
  result is much more general general as it holds under moment
  conditions only, whether a density exists or not for
  $\mathbb{P}_X$. Indeed, Proposition $5.2.2$ in \cite{BGL} requires
  existence of a density which, moreover, must be in the domain of the
  Dirichlet form associated with the Laguerre generator.
\end{rem}
\noindent
\textbf{Proof of Theorem \ref{DeBruijn2}, part I : interchange of derivative and integral}\\
We want to exchange the order between the
differentiation with respect to $\tau$ and the integration with
respect to $u$. For this purpose, we need to control uniformly in
$\tau$ the quantity 
$$\partial_{\tau}\big(g^{(\alpha,\lambda)}(\tau,u)\log\big(g^{(\alpha,\lambda)}(\tau,u)/\gamma_{\alpha,\lambda}(u)\big)\big).$$ By
standard computations, we have:
\begin{align*}
\bigg|\partial_{\tau}\bigg(g^{(\alpha,\lambda)}(\tau,u)\log\bigg(\dfrac{g^{(\alpha,\lambda)}(\tau,u)}{\gamma_{\alpha,\lambda}(u)}\bigg)\bigg)\bigg|\leq (I)+(II)+(III),
\end{align*}
with,
\begin{align*}
&(I)=|\partial_{\tau}\big(g^{(\alpha,\lambda)}(\tau,u)\big)|\big|\log\big(g^{(\alpha,\lambda)}(\tau,u)\big)\big|,\\
&(II)=|\partial_{\tau}\big(g^{(\alpha,\lambda)}(\tau,u)\big)|\big|\log\big(\gamma_{\alpha,\lambda}(u)\big)\big|,\\
&(III)=|\partial_{\tau}\big(g^{(\alpha,\lambda)}(\tau,u)\big)|.
\end{align*}
Let us deal with the first term, $(I)$. The others will follow similarly. By Proposition \ref{Transtech1}, we have:
\begin{align*}
|\partial_{\tau}\big(g^{(\alpha,\lambda)}(\tau,u)\big)|\leq P_{\alpha,\lambda,a,b}(u)g^{(\alpha,\lambda)}(\tau,u)+Q_{\alpha,\lambda,a,b}(u)h^{(\alpha,\lambda)}(\tau,u)+\frac{u}{\lambda a}k^{(\alpha,\lambda)}(\tau,u),
\end{align*}
with $P_{\alpha,\lambda,a,b}(.)$ and $Q_{\alpha,\lambda,a,b}(.)$ polynomials in $u$ of degree $1$ with positive coefficients. Moreover, by Lemma \ref{tech3}, we obtain:
\begin{align*}
(I)\leq \bigg[A_{X,\alpha,\lambda,a,b}+2\dfrac{\lambda u}{1-b}+|\alpha-1|\big|\log(u)\big|\bigg]\bigg[P_{\alpha,\lambda,a,b}(u)g^{(\alpha,\lambda)}(\tau,u)+Q_{\alpha,\lambda,a,b}(u)h^{(\alpha,\lambda)}(\tau,u)+\frac{u}{\lambda a}k^{(\alpha,\lambda)}(\tau,u)\bigg].
\end{align*}
Thus, if we can prove that $u^2k^{(\alpha,\lambda)}(\tau,u)$ is dominated uniformly in $\tau$ by an integrable function on $(0,+\infty)$, the first step will be done. By Lemma \ref{tech2}, we have:
\begin{align*}
u^2k^{(\alpha,\lambda)}(\tau,u)\leq C^3_{a,b,\lambda,\alpha}u^2\bigg(u^{\alpha+1}\exp\big(-\lambda \frac{u}{4(1-a)}\big)+u^{\alpha-1}\mathbb{E}[X^2\mathbb{I}_{\{X>\frac{u}{4}\}}]\bigg).
\end{align*}
The first term is clearly in $L^1\big((0,+\infty)\big)$. Moreover, for the second term, we have:
\begin{align*}
\int_0^{+\infty}u^{\alpha+1}\mathbb{E}[X^2\mathbb{I}_{\{X>\frac{u}{4}\}}]du&=\mathbb{E}\bigg[X^2\int_0^{+\infty}u^{\alpha+1}\mathbb{I}_{\{X>\frac{u}{4}\}}du\bigg],\\
&=\mathbb{E}\bigg[X^2\int_0^{4X}u^{\alpha+1}du\bigg],\\
&=\dfrac{4^{\alpha+2}}{\alpha+2}\mathbb{E}[X^{\alpha+4}]<+\infty.
\end{align*}
\\
\textbf{Proof of Theorem \ref{DeBruijn2}, part II : Integration by parts}\\
We proved that
one can interchange derivatives and integrals in the expression for
relative entropy to  obtain:
\begin{align*}
\dfrac{d}{d\tau}\bigg(D(X_{\tau}\|\gamma_{\alpha,\lambda})\bigg)&=\int_0^{+\infty}\partial_{\tau}\bigg(g^{(\alpha,\lambda)}(\tau,u)\log\bigg(\dfrac{g^{(\alpha,\lambda)}(\tau,u)}{\gamma_{\alpha,\lambda}(u)}\bigg)\bigg)du,\\
&=\int_0^{+\infty}\partial_{\tau}\big(g^{(\alpha,\lambda)}(\tau,u)\big)\log\bigg(\dfrac{g^{(\alpha,\lambda)}(\tau,u)}{\gamma_{\alpha,\lambda}(u)}\bigg)du,\\
&=-\frac{1}{\lambda\tau}\int_0^{+\infty}\partial_u\bigg(g^{(\alpha,\lambda)}(\tau,u)u\partial_u\bigg(\log\bigg(\dfrac{g^{(\alpha,\lambda)}(\tau,u)}{\gamma_{\alpha,\lambda}(u)}\bigg)\bigg)\bigg)\log\bigg(\dfrac{g^{(\alpha,\lambda)}(\tau,u)}{\gamma_{\alpha,\lambda}(u)}\bigg)du,
\end{align*}
where we have used Proposition \ref{Transtech1} in the last equality. Now, we perform cautiously an integration by parts. Let $R>1$ be big enough. For any $u\in[1/R,R]$, we have:
\begin{align*}
&\partial_u\bigg(g^{(\alpha,\lambda)}(\tau,u)u\partial_u\bigg(\log\bigg(\dfrac{g^{(\alpha,\lambda)}(\tau,u)}{\gamma_{\alpha,\lambda}(u)}\bigg)\bigg)\log\bigg(\dfrac{g^{(\alpha,\lambda)}(\tau,u)}{\gamma_{\alpha,\lambda}(u)}\bigg)\bigg)\\
& =\partial_u\bigg(g^{(\alpha,\lambda)}(\tau,u)u\partial_u\bigg(\log\bigg(\dfrac{g^{(\alpha,\lambda)}(\tau,u)}{\gamma_{\alpha,\lambda}(u)}\bigg)\bigg)\bigg)\\
&\times\log\bigg(\dfrac{g^{(\alpha,\lambda)}(\tau,u)}{\gamma_{\alpha,\lambda}(u)}\bigg)+g^{(\alpha,\lambda)}(\tau,u)u\bigg[\partial_u\bigg(\log\bigg(\dfrac{g^{(\alpha,\lambda)}(\tau,u)}{\gamma_{\alpha,\lambda}(u)}\bigg)\bigg)\bigg]^2.
\end{align*}
By the previous step, the first term is clearly integrable on $(0,+\infty)$. Let us compute explicitly the second term and study its integrability on $(0,+\infty)$. We have:
\begin{align*}
g^{(\alpha,\lambda)}(\tau,u)u\bigg[\partial_u\bigg(\log\bigg(\dfrac{g^{(\alpha,\lambda)}(\tau,u)}{\gamma_{\alpha,\lambda}(u)}\bigg)\bigg)\bigg]^2&=g^{(\alpha,\lambda)}(\tau,u)u\bigg[\dfrac{\partial_u\big(g^{(\alpha,\lambda)}\big)(\tau,u)}{g^{(\alpha,\lambda)}(\tau,u)}-\dfrac{\gamma'_{\alpha,\lambda}(u)}{\gamma_{\alpha,\lambda}(u)}\bigg]^2,\\
&=g^{(\alpha,\lambda)}(\tau,u)u\bigg[\dfrac{\partial_u\big(g^{(\alpha,\lambda)}\big)(\tau,u)}{g^{(\alpha,\lambda)}(\tau,u)}-\dfrac{\alpha-1}{u}+\lambda\bigg]^2.
\end{align*}
Moreover, by Proposition \ref{Transtech1}, we have:
\begin{align*}
\partial_u\big(g^{(\alpha,\lambda)}\big)(\tau,u)=\dfrac{\alpha-1}{u}g^{(\alpha,\lambda)}(\tau,u)-\dfrac{\lambda}{1-\tau}g^{(\alpha,\lambda)}(\tau,u)+h^{(\alpha,\lambda)}(\tau,u).
\end{align*}
Thus,
\begin{align*}
&
  g^{(\alpha,\lambda)}(\tau,u)u\bigg[\partial_u\bigg(\log\bigg(\dfrac{g^{(\alpha,\lambda)}(\tau,u)}{\gamma_{\alpha,\lambda}(u)}\bigg)\bigg)\bigg]^2\\
&=g^{(\alpha,\lambda)}(\tau,u)u\bigg[\dfrac{h^{(\alpha,\lambda)}(\tau,u)}{g^{(\alpha,\lambda)}(\tau,u)}-\dfrac{\lambda\tau}{1-\tau}\bigg]^2\\
&=u\dfrac{\big(h^{(\alpha,\lambda)}(\tau,u)\big)^2}{g^{(\alpha,\lambda)}(\tau,u)}-2uh^{(\alpha,\lambda)}(\tau,u)\frac{\lambda\tau}{1-\tau}+ug^{(\alpha,\lambda)}(\tau,u)\bigg(\dfrac{\lambda\tau}{1-\tau}\bigg)^2.
\end{align*}
The second and the third terms are clearly integrable (see Lemma \ref{tech2}). For the first term, we note that:
\begin{align*}
\int_0^{+\infty}ug^{(\alpha,\lambda)}(\tau,u)\dfrac{\big(h^{(\alpha,\lambda)}(\tau,u)\big)^2}{\big(g^{(\alpha,\lambda)}(\tau,u)\big)^2}du=I^{\tau}_{\gamma}(X_{\tau}),
\end{align*}
which is finite by Proposition \ref{FiniteFisher}.\\
\textbf{Proof of Theorem \ref{DeBruijn2}, part III : conclusion}\\
We need to study the limits at $0^+$ and at $+\infty$ of the following function:
\begin{align*}
u\longmapsto ug^{(\alpha,\lambda)}(\tau,u)\dfrac{\partial}{\partial u}\bigg(\log\bigg(\dfrac{g^{(\alpha,\lambda)}(\tau,u)}{\gamma_{\alpha,\lambda}(u)}\bigg)\bigg)\log\bigg(\dfrac{g^{(\alpha,\lambda)}(\tau,u)}{\gamma_{\alpha,\lambda}(u)}\bigg).
\end{align*}
The limits  exist by the proof
of Theorem 22. Moreover,
$u\rightarrow
\sqrt{ug^{(\alpha,\lambda)}(\tau,u)} \partial_u\big(\log\big(g^{(\alpha,\lambda)}(\tau,u)/\gamma_{\alpha,\lambda}(u)\big)\big)$
is a square-integrable function on $(0,+\infty)$. Thus, we need to
study the limits at $0^+$ and at $+\infty$ of the function:
\begin{align*}
u\longmapsto\sqrt{ug^{(\alpha,\lambda)}(\tau,u)}\log\bigg(\dfrac{g^{(\alpha,\lambda)}(\tau,u)}{\gamma_{\alpha,\lambda}(u)}\bigg).
\end{align*}
Using Lemma \ref{Asym0}, we have:
\begin{align*}
ug^{(\alpha,\lambda)}(\tau,u)\underset{u\rightarrow0^+}{\sim}C_{\tau,\alpha,\lambda}u^{\alpha},
\end{align*}
which ensures that the limit at $0^+$ of the previous function is
$0$. All that remains is to prove that
${\lim}_{{u\rightarrow+\infty}}
u^{3}g^{(\alpha,\lambda)}(\tau,u)=0$. This step is a
direct application
of the Lebesgue dominated convergence theorem. Indeed, by definition,
we have:
\begin{align*}
u^3g^{(\alpha,\lambda)}(\tau,u)=&\int_{0}^{+\infty}\dfrac{\lambda}{1-\tau}\big(\dfrac{1}{\tau}\big)^{\frac{\alpha-1}{2}}u^{\frac{\alpha-1}{2}+3}\exp\big(-\frac{\lambda}{1-\tau}u\big)\\
&\times\big(\dfrac{1}{x}\big)^{\frac{\alpha-1}{2}}\exp\big(-\dfrac{\lambda\tau}{1-\tau}x\big)I_{\alpha-1}\big(\dfrac{2\lambda\sqrt{ux\tau}}{1-\tau}\big)\mathbb{P}_X(dx).\\
\end{align*}
The almost everywhere convergence follows by asymptotic properties of the modified Bessel function of the first kind of order $\alpha-1$ at infinity. Moreover, denoting by $(1)$ the integrand, we have, $\mathbb{P}_X-a.e.$:
\begin{align*}
(1)&\leq C_{\alpha,\lambda,\tau} u^{\alpha+2}\exp\big(-\frac{\lambda}{1-\tau}u\big)\exp\big(-\dfrac{\lambda\tau}{1-\tau}x\big)\cosh\big(\dfrac{2\lambda\sqrt{ux\tau}}{1-\tau}\big),\\
&\leq C_{\alpha,\lambda,\tau} u^{\alpha+2}\exp\big(-\frac{\lambda}{1-\tau}(\sqrt{u}-\sqrt{x\tau})^2\big),\\
&\leq C^1_{\alpha,\lambda,\tau}+C^2_{\alpha,\lambda,\tau}x^{\alpha+2},
\end{align*}
which is integrable since $X$ admits moment of order $\alpha+4$. $\Box$\\
\\
As anticipated, we conclude with a proof of an integrated De Bruijn
identity. 
\begin{theo}\label{sec:de-bruijn-formula}
  Under the assumptions of Theorem \ref{DeBruijn2} : if $\mathbb{E}
  \left[ X \right] = \alpha/\lambda$ then  
  \begin{equation}
    \label{eq-IntBruijn}
    D(X\, || \, \gamma_{\alpha,\lambda}) = \int_0^1 \frac{1}{\lambda
      \tau}J_{st,\gamma}(X_{\tau}) d \tau.
  \end{equation}
\end{theo}
\begin{proof}
First of all, note that, by Propositions \ref{FiniteFisher} and \ref{CramerRao}, we can integrate $1/(\lambda\tau)J_{st,\gamma}(X_{\tau})$ over any compact subsets $[a,b]$ strictly contained in $[0,1]$. Thus, we get:
\begin{align*}
D(X_b\|\gamma_{\alpha,\lambda})-D(X_a\|\gamma_{\alpha,\lambda})=\int_a^b\frac{1}{\lambda \tau}J_{st,\gamma}(X_{\tau})d\tau.
\end{align*}
Now, we want to prove that:
\begin{align*}
\underset{b\rightarrow 1^-}{\lim}D(X_b\|\gamma_{\alpha,\lambda})=D(X\|\gamma_{\alpha,\lambda}),\quad
\underset{a\rightarrow 0^+}{\lim}D(X_a\|\gamma_{\alpha,\lambda})=0.
\end{align*}
Let us study the first limit. We know that $X_b\rightarrow X$ in distribution when $b\rightarrow 1^-$ by Lemma \ref{Inter}. Since the relative entropy is lower semicontinuous, we have that:
\begin{align*}
\underset{b\rightarrow 1^-}{\liminf}D(X_b\|\gamma_{\alpha,\lambda})\geq D(X\|\gamma_{\alpha,\lambda}).
\end{align*}
Now, by Lemma $6.2.13$ of \cite{DemZei}, we have the following representation for the relative entropy:
\begin{align*}
D(X\|\gamma_{\alpha,\lambda})=\underset{\phi\in C_b(\mathbb{R}_+)}{\sup}\bigg(\int_0^{+\infty}\phi(x)f_X(x)dx-\log\bigg(\int_0^{+\infty}\exp(\phi(x))\gamma_{\alpha,\lambda}(x)dx\bigg)\bigg).
\end{align*}
Following the beginning of the proof of Lemma $1.2$ in \cite{Carlen2}, using the fact that the Laguerre semigroup is a contraction on every $L^p\big(\mathbb{R}_+^*,\gamma_{\alpha,\lambda}(x)dx\big)$, for $p\geq 1$, and the fact that $P_t\big(C_b(\mathbb{R}_+)\big)\subset C_b(\mathbb{R}_+)$, we obtain:
\begin{align*}
D(X_b\|\gamma_{\alpha,\lambda})\leq D(X\|\gamma_{\alpha,\lambda}).
\end{align*} 
Thus, the first limit is proved. To finish, let us prove that $D(X_a\|\gamma_{\alpha,\lambda}) \to 0$ as $a \to 0$. Using 
Definition \ref{main} along with the known fact that entropy increases with
independent convolution, we deduce that
\begin{align*}
 H(\gamma({\alpha, \lambda})) -  H(X_a) \le H(\gamma({\alpha,
  \lambda})) - H((1-a)\gamma(\alpha, \lambda)) =  - \log (1-a)
\end{align*}
where in the last equality we used the known identity
$H(cY) = H(Y) + \log c$ for $c>0$.  Likewise, still using Definition \ref{main}
and the fact that the second summand in this representation is
positive, we see that
\begin{align*}
\ \mathbb{E} \left[ \log \gamma_{\alpha, \lambda} \right]  -
    \mathbb{E} \left[ \log X_a \right]  \le - \log(1-a). 
\end{align*}
Using the following relation,
\begin{align*}
D(X_a\|\gamma_{\alpha,\lambda})=H(\gamma_{\alpha,\lambda})-H(X_a)+(\alpha-1)\int_0^{+\infty}\big(\gamma_{\alpha,\lambda}(u)-g^{(\alpha,\lambda)}(\tau,u)\big)\log(u)du.
\end{align*} 
we obtain, 
\begin{align*}
0 \le  D(X_a\|\gamma_{\alpha,\lambda})  \le   -\alpha \log(1-a)
\end{align*}
and the conclusion follows.
\end{proof}
\begin{rem}\label{sec:shannon}
By fixing the logarithmic moment of $X$ to be equal to that of $\gamma_{\alpha,\lambda}$, we obtain the following straightforward identity:
\begin{align*}
H(\gamma_{\alpha,\lambda})-H(X)= \int_0^1 \frac{1}{\lambda
      \tau}J_{st,\gamma}(X_{\tau}) d \tau,
\end{align*}
where $H$ is the Shannon entropy.
\end{rem}


\section{A new proof of LSI for the gamma case with $\alpha\geq\frac{1}{2}$}\label{sec-LSI}


Before stating the main result of this section, we introduce some
notations. Let $\mathcal{B}\big((0,+\infty)\big)$ be the set of
bounded measurable functions defined on $(0,+\infty)$ and
$\mathcal{C}^1_b\big((0,+\infty)\big)$ be the set of continuously
differentiable functions with bounded derivatives up to order $1$. For
simplicity, we denote by
$\{P^{(\alpha,\lambda)}_{\tau},\ \tau\in(0,1)\}$ the Laguerre
semigroup after the time change $-\log(\tau)/\lambda$ acting on
functions in $\mathcal{B}\big((0,+\infty)\big)$. We recall the
well-known transition kernel of this semigroup with respect to the
Lebesgue measure:
\begin{align*}
p^{(\alpha,\lambda)}_{\tau}(x,u)=\dfrac{\lambda}{1-\tau}\bigg(\dfrac{u}{\tau x}\bigg)^{\frac{\alpha-1}{2}}\exp\bigg(-\dfrac{\lambda u}{1-\tau}\bigg)\exp\bigg(-\dfrac{\lambda\tau x}{1-\tau}\bigg)I_{\alpha-1}\bigg(\dfrac{2\lambda\sqrt{ux\tau}}{1-\tau}\bigg)
\end{align*}
\noindent
Moreover, it should be clear from the previous formula that the Laguerre semigroup is symmetric on $L^2\big(\mathbb{R}^*_+,\gamma_{\alpha,\lambda}(u)du\big)$.
\begin{lem}\label{MehInter}
Let $\alpha \geq \frac{1}{2}$. We have the following stochastic representations:
\begin{align}
&\forall f\in \mathcal{B}\big((0,+\infty)\big),\ P^{(\alpha,\lambda)}_{\tau}(f)(x)=\mathbb{E}\bigg[f\bigg((1-\tau)\gamma(\alpha-\frac{1}{2},\lambda)+\bigg((\sqrt{\tau}\sqrt{x}+\sqrt{\frac{1-\tau}{2\lambda}}Z\bigg)^2\bigg)\bigg],\\
&\forall f\in \mathcal{C}^1_b\big((0,+\infty)\big),\ \partial_x^{\sigma}\big(P^{(\alpha,\lambda)}_{\tau}(f)\big)(x)=\sqrt{\tau}\mathbb{E}\bigg[\dfrac{\big(\sqrt{\tau}\sqrt{x}+\sqrt{\frac{1-\tau}{2\lambda}}Z\big)}{(X^x_\tau)^{\frac{1}{2}}}\partial^{\sigma}(f)(X^x_\tau)\bigg]\label{eq:1},
\end{align}
with:
\begin{align*}
&\partial_x^{\sigma}(\phi)(x)=\sqrt{x}\phi'(x),\\
&X^x_\tau=(1-\tau)\gamma(\alpha-\frac{1}{2},\lambda)+\bigg(\sqrt{\tau}\sqrt{x}+\sqrt{\frac{1-\tau}{2\lambda}}Z\bigg)^2.
\end{align*}
\end{lem}
\begin{proof}:
Let $f$ be in $\mathcal{B}\big((0,+\infty)\big)$. Let $x\in(0,+\infty)$. Let $X$ be a random variable whose law is the Dirac measure at $x$. We denote by $g^{(\alpha,\lambda,x)}(\tau,u)$ the density of the gamma smart path built with such a $X$ denoted by $X^x_\tau$. By definition of the Laguerre semigroup and Definition \ref{main}, we readily have:
\begin{align*}
P^{(\alpha,\lambda)}_{\tau}(f)(x)=\int_0^{+\infty}f(u)g^{(\alpha,\lambda,x)}(\tau,u)du=\mathbb{E}[f\big(X^x_\tau\big)].
\end{align*}
Formula (\ref{eq:1}) follows by standard computations. 
\end{proof}
\noindent
We note in particular that:
\begin{align}
\mid{\dfrac{\big(\sqrt{\tau}\sqrt{x}+\sqrt{\frac{1-\tau}{2\lambda}}Z\big)}{(X^x_\tau)^{\frac{1}{2}}}\mid}\leq 1\label{eq:2}
\end{align}
\noindent
Using (\ref{eq:1}), we want to derive a simple bound for the Fisher information structure along the evolute $X_{\tau}$ involving the Fisher information structure of $X$. From it, we obtain the logarithmic Sobolev inequality for the gamma law which is known to hold for $\alpha\geq 1/2$ (\cite{Bakry}). This approach extends a classical Cauchy-Schwarz argument from the Gaussian to the gamma case. Let $f_X$ be the density of the random variable $X$.
\begin{prop}\label{LSI}
Let $\alpha\geq\frac{1}{2}$. Assume that $J_{st,\gamma}(X)<+\infty$ and that $f_X$ is smooth enough. Then, we have:
\begin{align*}
D(X\|\gamma_{\alpha, \lambda})\leq \frac{1}{\lambda}J_{st,\gamma}(X).
\end{align*}
\end{prop}
\begin{proof}
The density of $X_{\tau}$ is given by:
\begin{align*}
g^{(\alpha,\lambda)}(\tau,u)=\int_{0}^{+\infty}p^{(\alpha,\lambda)}_{\tau}(x,u)f_X(x)dx.
\end{align*}
By definition of the Fisher information structure, we have:
\begin{align*}
J_{st,\gamma}(X_{\tau})&=\int_0^{+\infty}ug^{(\alpha,\lambda)}(\tau,u)\bigg[\partial_u\big(\log(\dfrac{g^{(\alpha,\lambda)}(\tau,u)}{\gamma_{\alpha,\lambda}(u)})\big)\bigg]^2du,\\
&=\int_0^{+\infty}u\dfrac{\big(\partial_u\big(f^{(\alpha,\lambda)}(\tau,u)\big)\big)^2}{f^{(\alpha,\lambda)}(\tau,u)}\gamma_{\alpha,\lambda}(u)du,
\end{align*}
with $f^{(\alpha,\lambda)}(\tau,u)=g^{(\alpha,\lambda)}(\tau,u)/\gamma_{\alpha,\lambda}(u)$. Using duality and symmetry of the transition kernel of the Laguerre semigroup with respect to the measure $\gamma_{\alpha,\lambda}(u)du$, we have the following representation:
\begin{align*}
J_{st,\gamma}(X_{\tau})=\int_0^{+\infty}u\dfrac{\big(\partial_u\big(P^{(\alpha,\lambda)}_{\tau}(f_X/\gamma_{\alpha,\lambda})\big)\big)^2}{P^{(\alpha,\lambda)}_{\tau}(f_X/\gamma_{\alpha,\lambda})}\gamma_{\alpha,\lambda}(u)du.
\end{align*}  
At this point, we use the fact that $f_X$ is smooth enough in order to ensure that (\ref{eq:1}) is true for the function $f_X/\gamma_{\alpha,\lambda}$. We obtain:
\begin{align*}
J_{st,\gamma}(X_{\tau})&\leq \tau \int_0^{+\infty}\dfrac{\mathbb{E}\big[\mid{\partial^{\sigma}(f_X/\gamma_{\alpha,\lambda})(X^x_\tau)\mid}\big]^2}{P^{(\alpha,\lambda)}_{\tau}(f_X/\gamma_{\alpha,\lambda})}\gamma_{\alpha,\lambda}(u)du,\\
&\leq \tau\int_0^{+\infty}\dfrac{P^{(\alpha,\lambda)}_{\tau}\big(\mid{\partial^{\sigma}(f_X/\gamma_{\alpha,\lambda})\mid}\big)^2}{P^{(\alpha,\lambda)}_{\tau}(f_X/\gamma_{\alpha,\lambda})}\gamma_{\alpha,\lambda}(u)du,\\
&\leq \tau\int_0^{+\infty}P^{(\alpha,\lambda)}_{\tau}\bigg(\dfrac{\mid{\partial^{\sigma}(f_X/\gamma_{\alpha,\lambda})\mid}^2}{f_X/\gamma_{\alpha,\lambda}}\bigg)\gamma_{\alpha,\lambda}(u)du,\\
&\leq \tau\int_0^{+\infty}\dfrac{\mid{\partial^{\sigma}(f_X/\gamma_{\alpha,\lambda})\mid}^2}{f_X/\gamma_{\alpha,\lambda}}\gamma_{\alpha,\lambda}(u)du,\\
&\leq \tau\int_0^{+\infty}u\dfrac{\big(\partial_u\big(f_X/\gamma_{\alpha,\lambda}\big)\big)^2}{f_X/\gamma_{\alpha,\lambda}}\gamma_{\alpha,\lambda}(u)du=\tau J_{st,\gamma}(X),
\end{align*}
where we have used successively (\ref{eq:1}), (\ref{eq:2}), Cauchy-Schwarz inequality and the invariance of the gamma measure along the Laguerre dynamic. Using De Bruijn identity (\ref{eq-IntBruijn}) and integrating the previous inequality, we obtain the following form of the logarithmic Sobolev inequality for the gamma measure $\gamma_{\alpha,\lambda}$ with $\alpha\geq 1/2$:
\begin{align*}
D(X\|\gamma_{\alpha, \lambda})\leq \frac{1}{\lambda}J_{st,\gamma}(X).
\end{align*}
\end{proof}
\begin{rem}
Through the course of the previous proof, we have obtained the information theoretical inequality,  $J_{st,\gamma}(X_{\tau})\leq\tau J_{st,\gamma}(X)$, which should be compared to the following classical inequality obtained thanks to the Blachman-Stam inequality:
\begin{align*}
\forall t\in(0,1),\ J_{st}(\sqrt{t}X+\sqrt{1-t}Z)\leq tJ_{st}(X),
\end{align*}
where $Z$ is a standard normal random variable independent of $X$ and $J_{st}(.)$ denotes the classical standardized Fisher information structure associated with the Gaussian law.
\end{rem}


\section{A new HSI-type inequality}\label{sec-HSI}
In this section, we develop the tools needed to obtain an HSI-type inequality linking relative Entropy, Stein discrepancy and standardized Fisher information in the spirit of the ones obtained in \cite{LNP}. For this purpose, we define a new type of Stein kernel based on a rewriting of the Laguerre operator, $\mathcal{L}_{\alpha,\lambda}$. This rewriting is justified by the following lemma which provides Bismut-type representation for iterated actions of the operator $\partial_x^{\sigma}$ previously introduced on the Laguerre semigroup.
\begin{lem}\label{Bismut}
Let $\alpha\geq 1/2$. Then, we have:
\begin{align}
\forall f\in \mathcal{B}\big((0,+\infty)\big),\  \partial_x^{\sigma} \big(P^{(\alpha,\lambda)}_{\tau}(f)\big)(x)=\sqrt{\frac{\lambda}{2}}\sqrt{\dfrac{\tau}{1-\tau}}\mathbb{E}\big[Zf(X^x_\tau)\big].\label{eq-Bismut}
\end{align}
Moreover, for any integer $k\geq 1$, we have:
\begin{align*}
\forall f\in \mathcal{B}\big((0,+\infty)\big),\ (\partial_x^{\sigma})^k\big(P^{(\alpha,\lambda)}_{\tau}(f)\big)(x)=\bigg(\frac{\lambda}{2}\bigg)^{\frac{k}{2}}\bigg(\dfrac{\tau}{1-\tau}\bigg)^{\frac{k}{2}}\mathbb{E}\big[H_k(Z)f(X^x_\tau)\big],
\end{align*}
where $H_k(.)$ denotes the k-th Hermite polynomial.
\end{lem}
\begin{proof} Let $f\in \mathcal{C}^1_b\big((0,+\infty)\big)$. By Lemma \ref{MehInter}, we have:
\begin{align*}
\partial_x^{\sigma}\big(P^{(\alpha,\lambda)}_{\tau}(f)\big)(x)&=\sqrt{\tau}\mathbb{E}\bigg[\dfrac{\big(\sqrt{\tau}\sqrt{x}+\sqrt{\frac{1-\tau}{2\lambda}}Z\big)}{(X^x_\tau)^{\frac{1}{2}}}\partial^{\sigma}(f)(X^x_\tau)\bigg],\\
&=\sqrt{\tau}\mathbb{E}\bigg[\big(\sqrt{\tau}\sqrt{x}+\sqrt{\frac{1-\tau}{2\lambda}}Z\big)f'(X^x_\tau)\bigg],\\
&=\sqrt{\frac{\lambda}{2}}\sqrt{\dfrac{\tau}{1-\tau}}\mathbb{E}\big[Zf(X^x_\tau)\big],
\end{align*}
where we have performed a Gaussian integration by parts in order to
get the last line. We proceed by density to extend this relation to
functions in $ \mathcal{B}\big((0,+\infty)\big)$. The general case is
obtained by a recursive argument together with standard relations
regarding Hermite polynomials.
\end{proof}
\noindent
Recall that the Laguerre operator is given by the following formula on
sufficiently smooth functions:
\begin{align*}
\mathcal{L}_{\alpha,\lambda}(f)(u)=u\dfrac{d^2 f}{du^2}(u)+(\alpha-\lambda u)\dfrac{d f}{du}(u).
\end{align*}
Using the operator $\partial_x^{\sigma}$ this can be rewritten as:
\begin{align*}
\mathcal{L}_{\alpha,\lambda}(f)(u)=\big(\partial_u^{\sigma}\big)^2(f)(u)+(\alpha-\frac{1}{2}-\lambda u)\dfrac{d f}{du}(u).
\end{align*}
Thus, it is natural to introduce the following Stein kernel for probability measure on $(0,+\infty)$.
\begin{defi}\label{Stein}
Let $X$ be a random variable with values in $(0,+\infty)$. Then, we define the Stein kernel of $X$, $\tau_X(.)$, for every smooth test function by:
\begin{align*}
\mathbb{E}[(\lambda X-\alpha+\frac{1}{2})\phi(X)]=\mathbb{E}[\tau_X(X)\sqrt{X}(\partial_x^{\sigma})^*(\phi)(X)],
\end{align*}
where $(\partial_x^{\sigma})^*(.)=\frac{d}{dx}(\sqrt{x}.)$. In particular, we have:
\begin{align}
\mathbb{E}[(\lambda X-\alpha+\frac{1}{2})\phi'(X)]=\mathbb{E}[\tau_X(X)(\partial_x^{\sigma})^2(\phi)(X)]\label{Stein}.
\end{align}
\end{defi}
\begin{rem}
When $\mathbb{E}[X]=\alpha/\lambda$, we note that $\mathbb{E}[\tau_X(X)]=1$. In particular, the Stein discrepancy is exactly the variance of $\tau_X(X)$.
\end{rem}
\noindent
Before stating the main result of this section, we introduce a
fundamental representation of the standardized Fisher information
structure in terms of the previously defined Stein kernel. This
representation should be compared to the one obtained in Proposition
$2.4$ (iii) of \cite{LNP}. At the core of its proof stand the
Bismut-type representation of
$\partial_x^{\sigma} (P^{(\alpha,\lambda)}_{\tau}(f))$ as well as the
intertwining relation (\ref{eq:1}). We assume that the random variable
$X$ has a density $f_X$ such that $f=f_X/\gamma_{\alpha,\lambda}$ is
smooth enough for the different analytical arguments to
hold. Moreover, we assume that the Stein kernel of $X$ exists.
\begin{prop}\label{FisherRep}
Let $\alpha\geq 1/2$. We have:
\begin{align*}
J_{st,\gamma}(X_{\tau})=\sqrt{\frac{\lambda}{2}}\dfrac{\tau}{\sqrt{1-\tau}}\mathbb{E}[(\tau_X(X)-1)Z\operatorname{\mathcal{V}}(X_\tau)\partial^{\sigma}(v_\tau)(X_\tau)],
\end{align*}
with 
\begin{align*}
&v_\tau=\log\big(P^{(\alpha,\lambda)}_{\tau}(f_X/\gamma_{\alpha,\lambda})\big),\\
&\operatorname{\mathcal{V}}(X_\tau)=\dfrac{\big(\sqrt{\tau}\sqrt{X}+\sqrt{\frac{1-\tau}{2\lambda}}Z\big)}{\sqrt{(1-\tau)\gamma(\alpha-\frac{1}{2},\lambda)+\big(\sqrt{\tau}\sqrt{X}+\sqrt{\frac{1-\tau}{2\lambda}}Z\big)^2}}.
\end{align*}
\end{prop}
\begin{proof} Recall that we have the following representation for the
  Fisher information structure along the smart path:
\begin{align}
J_{st,\gamma}(X_{\tau})=\int_0^{+\infty}u\dfrac{\big(\partial_u\big(P^{(\alpha,\lambda)}_{\tau}(f_X/\gamma_{\alpha,\lambda})\big)\big)^2}{P^{(\alpha,\lambda)}_{\tau}(f_X/\gamma_{\alpha,\lambda})}\gamma_{\alpha,\lambda}(u)du.\label{RepFisher}
\end{align}
Since the Laguerre generator is a diffusion, it satisfies the following integration by parts formula on smooth functions $\phi,\psi$ from $(0,+\infty)$ to $\mathbb{R}$:
\begin{align*}
\int_{0}^{+\infty}\phi(u)\mathcal{L}_{\alpha,\lambda}(\psi)(u)\gamma_{\alpha,\lambda}(u)du=-\int_{0}^{+\infty}\partial_u^{\sigma}(\phi)\partial_u^{\sigma}(\psi)(u)\gamma_{\alpha,\lambda}(u)du.
\end{align*}
Thus, we have:
\begin{align*}
J_{st,\gamma}(X_{\tau})&=-\int_0^{+\infty}\mathcal{L}_{\alpha,\lambda}\big(P^{(\alpha,\lambda)}_{\tau}(v_\tau)\big)(x)f_X(x)dx,\\
&=-\int_0^{+\infty}\big[\big(\partial_x^{\sigma}\big)^2(P^{(\alpha,\lambda)}_{\tau}(v_\tau))(x)+(\alpha-\frac{1}{2}-\lambda x)\partial_x(P^{(\alpha,\lambda)}_{\tau}(v_\tau))(x)\big]f_X(x)dx,\\
&=-\int_0^{+\infty}\big[1-\tau_X(x)\big]\big(\partial_x^{\sigma}\big)^2(P^{(\alpha,\lambda)}_{\tau}(v_\tau))(x)f_X(x)dx,
\end{align*}
where we have used the definition of the Stein kernel on the function $\partial_x(P^{(\alpha,\lambda)}_{\tau}(v_\tau))$. Now, thanks to Bismut formula (\ref{eq-Bismut}), we have:
\begin{align*}
\partial_x^{\sigma}(P^{(\alpha,\lambda)}_{\tau}(v_\tau))(x)=\sqrt{\frac{\lambda}{2}}\sqrt{\dfrac{\tau}{1-\tau}}\mathbb{E}\big[Zv_\tau(X^x_\tau)\big].
\end{align*}
Moreover, using the intertwining relation (\ref{eq:1}), we obtain:
\begin{align*}
\big(\partial_x^{\sigma}\big)^2(P^{(\alpha,\lambda)}_{\tau}(v_\tau))(x)=\sqrt{\frac{\lambda}{2}}\dfrac{\tau}{\sqrt{1-\tau}}\mathbb{E}\bigg[Z\dfrac{\big(\sqrt{\tau}\sqrt{x}+\sqrt{\frac{1-\tau}{2\lambda}}Z\big)}{(X^x_\tau)^{\frac{1}{2}}}\partial^\sigma(v_\tau)(X^x_\tau)\bigg].
\end{align*}
The conclusion follows by integrating out with respect to the law of $X$.
\end{proof}
\noindent
We are now ready to state the main result of this section.
\begin{theo}\label{HSI}
Let $\alpha\geq 1/2$. Let $X$ be a strictly positive random variable with finite $\alpha+4$ moments and such that $J_{st,\gamma}(X)<+\infty$, $\mathbb{E}[(\tau_X(X)-1)^2]<+\infty$ and $\mathbb{E}[X]=\alpha/\lambda$. We have:
\begin{align*}
D(X\|\gamma_{\alpha,\lambda})\leq \frac{1}{2}\mathbb{E}[(\tau_X(X)-1)^2]\log\bigg[1+\frac{2}{\lambda}\dfrac{J_{st,\gamma}(X)}{\mathbb{E}[(\tau_X(X)-1)^2]}\bigg].
\end{align*}
\end{theo}
\begin{rem}
\begin{itemize}
\item This inequality should be compared to the one of Proposition $4.3$ in \cite{LNP}. In particular, the Stein kernel used in the definition of the Stein discrepancy is different from the one we use here. In contrast to the proof of this result, we do not use the operators $\Gamma_2$ and $\Gamma_3$ from the $\Gamma$-calculus of \cite{BGL}. This should emphasize the strength of stochastic representations such as the ones in Lemmata \ref{MehInter} and \ref{Bismut}. The method of the proof is similar to the one of Theorem $2.2$ of \cite{LNP} in the Gaussian case.
\item Moreover, it is important to note that this functional inequality improves upon the classical logarithmic Sobolev inequality (Proposition \ref{LSI}) as in the Gaussian case.
\end{itemize}
\end{rem}
\begin{proof}
From De Bruijn identity of Theorem \ref{sec:de-bruijn-formula}, we have:
\begin{align*}
D(X\|\gamma_{\alpha,\lambda}) &= \int_0^1 \frac{1}{\lambda \tau}J_{st,\gamma}(X_{\tau}) d \tau,\\
&= \int_0^u \frac{1}{\lambda \tau}J_{st,\gamma}(X_{\tau}) d \tau+\int_u^1 \frac{1}{\lambda \tau}J_{st,\gamma}(X_{\tau}) d \tau.
\end{align*}
For times closed to 1, we use the following bound (from the proof of Proposition \ref{LSI}):
\begin{align}
\int_u^1 \frac{1}{\lambda \tau}J_{st,\gamma}(X_{\tau}) d \tau\leq \int_u^1 \frac{\tau}{\lambda \tau}J_{st,\gamma}(X) d \tau\leq \frac{1}{\lambda}J_{st,\gamma}(X)(1-u).\label{eq-Term1}
\end{align}
Moreover, thanks to Proposition \ref{FisherRep}, (\ref{eq:2}), and Cauchy-Schwarz inequality, we have:
\begin{align*}
J_{st,\gamma}(X_{\tau})\leq \sqrt{\frac{\lambda}{2}}\dfrac{\tau}{\sqrt{1-\tau}} \mathbb{E}[(\tau_X(X)-1)^2]^{\frac{1}{2}}\mathbb{E}[\mid\partial^{\sigma}(v_\tau)(X_\tau)\mid^2]^{\frac{1}{2}}.
\end{align*}
But,
\begin{align*}
\mathbb{E}[\mid\partial^{\sigma}(v_\tau)(X_\tau)\mid^2]&=\int_{0}^{+\infty}P^{(\alpha,\lambda)}_\tau(\mid\partial^{\sigma}(v_\tau)\mid^2)(x)f_X(x)dx,\\
&=\int_{0}^{+\infty}\mid\partial^{\sigma}(v_\tau)(u)\mid^2P^{(\alpha,\lambda)}_\tau(f_X/\gamma_{\alpha,\lambda})(u)\gamma_{\alpha,\lambda}(u)du,\\
&=J_{st,\gamma}(X_{\tau}).
\end{align*}
Thus, we have the following bound:
\begin{align*}
J_{st,\gamma}(X_{\tau})\leq \frac{\lambda}{2}\dfrac{\tau^2}{1-\tau} \mathbb{E}[(\tau_X(X)-1)^2],
\end{align*}
which provides the following estimate for small times:
\begin{align*}
\int_0^u \frac{1}{\lambda \tau}J_{st,\gamma}(X_{\tau}) d \tau&\leq \frac{1}{2}\mathbb{E}[(\tau_X(X)-1)^2]\int_0^u \dfrac{\tau}{1-\tau}d \tau,\\
&\leq -\frac{1}{2}\mathbb{E}[(\tau_X(X)-1)^2](u+\log(1-u)).
\end{align*}
The result follows with an optimisation in $u\in(0,1)$.
\end{proof}


\section{Appendix}\label{sec-Append}


\begin{prop}\label{Transtech1}
Let $\alpha\geq 1/2$. We have for every $(\tau,u)\in (0,1)\times(0,+\infty)$:
\begin{align}
\partial_u\big(g^{(\alpha,\lambda)}\big)(\tau,u)&=\dfrac{\alpha-1}{u}g^{(\alpha,\lambda)}(\tau,u)-\dfrac{\lambda}{1-\tau}g^{(\alpha,\lambda)}(\tau,u)+h^{(\alpha,\lambda)}(\tau,u),\\ 
-\lambda\tau\dfrac{\partial g^{(\alpha,\lambda)}}{\partial \tau}(\tau,u)&=-\dfrac{\partial}{\partial u}\bigg(g^{(\alpha,\lambda)}(\tau,u)u\dfrac{\partial}{\partial u}\log\bigg(\dfrac{g^{(\alpha,\lambda)}(\tau,u)}{\gamma_{\alpha,\lambda}(u)}\bigg)\bigg),\\
&=g^{(\alpha,\lambda)}(\tau,u)\bigg(u\dfrac{\lambda^2\tau}{(1-\tau)^2}-\dfrac{\lambda\alpha\tau}{1-\tau}\bigg)+h^{(\alpha,\lambda)}(\tau,u)\bigg(\alpha-u\lambda\dfrac{1+\tau}{1-\tau}\bigg)\nonumber\\
&+uk^{(\alpha,\lambda)}(\tau,u),
\end{align}
with,
\begin{align*}
h^{(\alpha,\lambda)}(\tau,u)&=\dfrac{\lambda^2}{(1-\tau)^2}\big(\dfrac{1}{\tau}\big)^{\frac{\alpha-2}{2}}u^{\frac{\alpha-2}{2}}\exp\big(-\frac{\lambda}{1-\tau}u\big)\\
&\times\int_0^{+\infty}\big(\dfrac{1}{x}\big)^{\frac{\alpha-2}{2}}\exp\big(-\dfrac{\lambda\tau}{1-\tau}x\big)I_{\alpha}\big(\dfrac{2\lambda\sqrt{ux\tau}}{1-\tau}\big)\mathbb{P}_X(dx),\\
k^{(\alpha,\lambda)}(\tau,u)&=\dfrac{\lambda^3}{(1-\tau)^3}\big(\dfrac{1}{\tau}\big)^{\frac{\alpha-3}{2}}u^{\frac{\alpha-3}{2}}\exp\big(-\frac{\lambda}{1-\tau}u\big)\\
&\times\int_0^{+\infty}\big(\dfrac{1}{x}\big)^{\frac{\alpha-3}{2}}\exp\big(-\dfrac{\lambda\tau}{1-\tau}x\big)I_{\alpha+1}\big(\dfrac{2\lambda\sqrt{ux\tau}}{1-\tau}\big)\mathbb{P}_X(dx).
\end{align*}
\end{prop}
\begin{proof}
We begin by computing the first partial derivative of $g^{(\alpha,\lambda)}(.,.)$ with respect to $u$. In order to do so, we need to justify properly the interchange of derivative and integral since:
\begin{align*}
g^{(\alpha,\lambda)}(\tau,u)=\dfrac{\lambda}{1-\tau}\bigg(\dfrac{u}{\tau}\bigg)^{\frac{\alpha-1}{2}}\exp\bigg(-\dfrac{\lambda u}{1-\tau}\bigg)\int_0^{+\infty}(\frac{1}{x})^{\frac{\alpha-1}{2}}\exp\bigg(-\dfrac{\lambda\tau x}{1-\tau}\bigg)I_{\alpha-1}\bigg(\dfrac{2\lambda\sqrt{ux\tau}}{1-\tau}\bigg)\mathbb{P}_X(dx).
\end{align*}
Let $\operatorname{K}=[a,b]$ be a compact set strictly contained in $(0,+\infty)$. We need to control uniformly the following quantity for $u\in\operatorname{K}$:
\begin{align*}
(I)=\vert(\frac{1}{x})^{\frac{\alpha-1}{2}}\exp\bigg(-\dfrac{\lambda\tau x}{1-\tau}\bigg)\partial_u\big(I_{\alpha-1}\bigg(\dfrac{2\lambda\sqrt{ux\tau}}{1-\tau}\bigg)\big)\vert
\end{align*}
Using the fact that $I^{(1)}_{\alpha-1}(z)=I_{\alpha}(z)+I_{\alpha-1}(z)(\alpha-1)/z$, we obtain the following straightforward bound:
\begin{align*}
(I)&\leq (\frac{1}{x})^{\frac{\alpha-1}{2}}\exp\bigg(-\dfrac{\lambda\tau x}{1-\tau}\bigg)\dfrac{\lambda\sqrt{x\tau}}{(1-\tau)\sqrt{u}}\{I_{\alpha}\bigg(\dfrac{2\lambda\sqrt{ux\tau}}{1-\tau}\bigg)\\
&+\dfrac{\vert\alpha-1\vert(1-\tau)}{2\lambda\sqrt{ux\tau}}I_{\alpha-1}\bigg(\dfrac{2\lambda\sqrt{ux\tau}}{1-\tau}\bigg)\},\\
&\leq C_{\operatorname{K},\tau,\alpha,\lambda}\{(\frac{1}{x})^{\frac{\alpha-2}{2}}\exp\bigg(-\dfrac{\lambda\tau x}{1-\tau}\bigg)I_{\alpha}\bigg(\dfrac{2\lambda\sqrt{ux\tau}}{1-\tau}\bigg)\\
&+(\frac{1}{x})^{\frac{\alpha-1}{2}}\exp\bigg(-\dfrac{\lambda\tau x}{1-\tau}\bigg)I_{\alpha-1}\bigg(\dfrac{2\lambda\sqrt{ux\tau}}{1-\tau}\bigg)\}.
\end{align*}
In order to deal with the terms involving the modified Bessel functions of the first kind, we use the right-hand side of inequality $(6.25)$ from \cite{INEQ} which holds for $\alpha>1/2$. We obtain:
\begin{align*}
(I)&\leq C'_{\operatorname{K},\tau,\alpha,\lambda}\{x\exp\bigg(-\dfrac{\lambda\tau x}{1-\tau}\bigg)\cosh\bigg(\dfrac{2\lambda\sqrt{ux\tau}}{1-\tau}\bigg)+\exp\bigg(-\dfrac{\lambda\tau x}{1-\tau}\bigg)\cosh\bigg(\dfrac{2\lambda\sqrt{ux\tau}}{1-\tau}\bigg)\},\\
&\leq C'_{\operatorname{K},\tau,\alpha,\lambda}\{x\exp\bigg(-\dfrac{\lambda\tau x}{1-\tau}\bigg)\cosh\bigg(\dfrac{2\lambda\sqrt{bx\tau}}{1-\tau}\bigg)+\exp\bigg(-\dfrac{\lambda\tau x}{1-\tau}\bigg)\cosh\bigg(\dfrac{2\lambda\sqrt{bx\tau}}{1-\tau}\bigg)\}.
\end{align*}
The previous bound is clearly integrable with respect to the probability measure $\mathbb{P}_X$ on $(0,+\infty)$. We deal with the case $\alpha=1/2$ using the explicit expression for the modified Bessel functions of the first kind of order $-1/2$. We can thus perform the computations for the first partial derivative of $g^{(\alpha,\lambda)}(.,.)$ with respect to $u$. We have:
\begin{align*}
\partial_u(g^{(\alpha,\lambda)})(\tau,u)&=\dfrac{\lambda}{1-\tau}\big(\dfrac{1}{\tau}\big)^{\frac{\alpha-1}{2}}(\frac{\alpha-1}{2})u^{\frac{\alpha-1}{2}-1}\exp\big(-\frac{\lambda}{1-\tau}u\big)\\
&\times\int_{0}^{+\infty}\big(\dfrac{1}{x}\big)^{\frac{\alpha-1}{2}}\exp\big(-\dfrac{\lambda\tau}{1-\tau}x\big)I_{\alpha-1}\big(\dfrac{2\lambda\sqrt{ux\tau}}{1-\tau}\big)\mathbb{P}_X(dx)\\
&-\dfrac{\lambda}{1-\tau}\dfrac{\lambda}{1-\tau}\big(\dfrac{1}{\tau}\big)^{\frac{\alpha-1}{2}}u^{\frac{\alpha-1}{2}}\exp\big(-\frac{\lambda}{1-\tau}u\big)\\
&\times\int_{0}^{+\infty}\big(\dfrac{1}{x}\big)^{\frac{\alpha-1}{2}}\exp\big(-\dfrac{\lambda\tau}{1-\tau}x\big)I_{\alpha-1}\big(\dfrac{2\lambda\sqrt{ux\tau}}{1-\tau}\big)\mathbb{P}_X(dx)\\
&+\dfrac{\lambda}{1-\tau}\big(\dfrac{1}{\tau}\big)^{\frac{\alpha-1}{2}}u^{\frac{\alpha-1}{2}}\exp\big(-\frac{\lambda}{1-\tau}u\big)\\
&\times\int_{0}^{+\infty}\big(\dfrac{1}{x}\big)^{\frac{\alpha-1}{2}}\exp\big(-\dfrac{\lambda\tau}{1-\tau}x\big)\dfrac{2\lambda\sqrt{x\tau}}{1-\tau}\dfrac{1}{2\sqrt{u}}I^{(1)}_{\alpha-1}\big(\dfrac{2\lambda\sqrt{ux\tau}}{1-\tau}\big)\mathbb{P}_X(dx),\\
&=\frac{\alpha-1}{2u}g^{(\alpha,\lambda)}(\tau,u)-\frac{\lambda}{1-\tau}g^{(\alpha,\lambda)}(\tau,u)\\
&+\dfrac{\alpha-1}{2u}\dfrac{\lambda}{1-\tau}\big(\dfrac{1}{\tau}\big)^{\frac{\alpha-1}{2}}u^{\frac{\alpha-1}{2}}\exp\big(-\frac{\lambda}{1-\tau}u\big)\\
&\times\int_0^{+\infty}\big(\dfrac{1}{x}\big)^{\frac{\alpha-1}{2}}\exp\big(-\dfrac{\lambda\tau}{1-\tau}x\big)I_{\alpha-1}\big(\dfrac{2\lambda\sqrt{ux\tau}}{1-\tau}\big)\mathbb{P}_X(dx)\\
&+\dfrac{\lambda^2}{(1-\tau)^2}\big(\dfrac{1}{\tau}\big)^{\frac{\alpha-2}{2}}u^{\frac{\alpha-2}{2}}\exp\big(-\frac{\lambda}{1-\tau}u\big)\\
&\times\int_0^{+\infty}\big(\dfrac{1}{x}\big)^{\frac{\alpha-2}{2}}\exp\big(-\dfrac{\lambda\tau}{1-\tau}x\big)I_{\alpha}\big(\dfrac{2\lambda\sqrt{ux\tau}}{1-\tau}\big)\mathbb{P}_X(dx).
\end{align*}
where we have used the fact that $I^{(1)}_{\alpha-1}(z)=I_{\alpha}(z)+I_{\alpha-1}(z)(\alpha-1)/z$. Therefore, we have:
\begin{align*}
\partial_u\big(g^{(\alpha,\lambda)}\big)(\tau,u)=\dfrac{\alpha-1}{u}g^{(\alpha,\lambda)}(\tau,u)-\dfrac{\lambda}{1-\tau}g^{(\alpha,\lambda)}(\tau,u)+h^{(\alpha,\lambda)}(\tau,u).
\end{align*}
By a similar reasoning, we obtain relations $(16)$ and $(17)$.
\end{proof}

\begin{lem}\label{Asym0}
Let $\alpha>0$. For $\tau\in(0,1)$, we have:
\begin{align*}
g^{(\alpha,\lambda)}(\tau,u)\underset{u\rightarrow 0^+}{\sim} C_{\tau,\alpha,\lambda}u^{\alpha-1},
\end{align*}
where $C_{\tau,\alpha,\lambda}$ is some strictly positive constant.
\end{lem}
\begin{proof}
Let $\mu>0$. The Laplace transform
of $g^{(\alpha,\lambda)}(\tau,.)$ is given by: 
\begin{align*}
L(g^{(\alpha,\lambda)}(\tau,.))(\mu)=\dfrac{1}{\big(1+\frac{\mu}{\lambda}(1-\tau)\big)^{\alpha}}L\big(\mathbb{P}_X\big)\big(\dfrac{\mu \tau}{1+\frac{\mu}{\lambda}(1-\tau)}\big).
\end{align*}
Since $L(\mathbb{P}_X)$ is in $C^{0}\big(]0,+\infty[\big)$, we have:
\begin{align*}
L(g^{(\alpha,\lambda)}(\tau,.))(\mu)\underset{\mu\rightarrow+\infty}{\sim}\dfrac{\lambda^{\alpha}}{(1-\tau)^{\alpha}}L\big(\mathbb{P}_X\big)\big(\dfrac{\lambda\tau}{1-\tau}\big)\dfrac{1}{\mu^{\alpha}}.
\end{align*}
Thus, by a classical Tauberian Theorem (see Chapter XIII.5, Theorem $3$ of \cite{Feller2}), we have:
\begin{align*}
g^{(\alpha,\lambda)}(\tau,u)\underset{u\rightarrow 0^+}{\sim}\dfrac{\lambda^{\alpha}}{(1-\tau)^{\alpha}}L\big(\mathbb{P}_X\big)\big(\dfrac{\lambda\tau}{1-\tau}\big)\dfrac{u^{\alpha}}{\Gamma(\alpha+1)},
\end{align*}
where $g^{(\alpha,\lambda)}(\tau,.)$ is the cumulative distribution function of $X_\tau$. In order to conclude, we need to know if $g^{\alpha,\lambda}(\tau,.)$ is monotone in a right neighborhood of $0$. Note that monotony properties of $g^{\alpha,\lambda}(\tau,.)$ can be deduced from those of $g^{(\alpha,\lambda)}(\tau,x,.)$. Moreover, since $I_{\nu}(z)\underset{0^+}{\sim}(z/2)^{\nu}1/\Gamma(\nu+1)$, we have:
\begin{align*}
g^{(\alpha,\lambda)}(\tau,x,u)\underset{u\rightarrow 0^+}{\sim}C_{\tau,\alpha,\lambda,x}u^{\alpha-1}\exp\big(\dfrac{-\lambda u}{1-\tau}\big),
\end{align*}
for some constant $C_{\tau,\alpha,\lambda,x}>0$. This implies that
$g^{\alpha,\lambda}(\tau,.)$ is monotone in a right neighborhood
of $0$. Thus,
\begin{align*}
g^{\alpha,\lambda}(\tau,u)\underset{u\rightarrow 0^+}{\sim}\dfrac{\alpha\lambda^{\alpha}}{(1-\tau)^{\alpha}}L\big(\mathbb{P}_X\big)\big(\dfrac{\lambda\tau}{1-\tau}\big)\dfrac{u^{\alpha-1}}{\Gamma(\alpha+1)},
\end{align*}
which concludes the proof.
\end{proof}
\begin{lem}\label{tech2}
Let $\alpha\geq 1/2$. Let $X$ be an almost surely positive random variable with first and second moments finite. Then, we have, for every $(\tau,u)\in(a,b)\times(0,+\infty)$:
\begin{align*}
&g^{(\alpha,\lambda)}(\tau,u)\leq C^1_{a,b,\lambda,\alpha}\bigg(u^{\alpha-1}\exp\big(-\lambda \frac{u}{4(1-a)}\big)+u^{\alpha-1}\int_{\frac{u}{4}}^{+\infty}\mathbb{P}_X(dx)\bigg),\\
&h^{(\alpha,\lambda)}(\tau,u)\leq C^2_{a,b,\lambda,\alpha}\bigg(u^{\alpha}\exp\big(-\lambda \frac{u}{4(1-a)}\big)+u^{\alpha-1}\mathbb{E}[X\mathbb{I}_{\{X>\frac{u}{4}\}}]\bigg),\\
&k^{(\alpha,\lambda)}(\tau,u)\leq C^3_{a,b,\lambda,\alpha}\bigg(u^{\alpha+1}\exp\big(-\lambda \frac{u}{4(1-a)}\big)+u^{\alpha-1}\mathbb{E}[X^2\mathbb{I}_{\{X>\frac{u}{4}\}}]\bigg),
\end{align*}
with $C^i_{a,b,\lambda,\alpha}$, $i\in\{1,2,3\}$, some strictly positive constants.
\end{lem}
\begin{proof}
By definition, for every $(\tau,u)\in (a,b)\times(0,+\infty)$, we have:
\begin{align*}
g^{(\alpha,\lambda)}(\tau,u)=&\dfrac{\lambda}{1-\tau}\big(\dfrac{1}{\tau}\big)^{\frac{\alpha-1}{2}}u^{\frac{\alpha-1}{2}}\exp\big(-\frac{\lambda}{1-\tau}u\big)\\
&\times\int_{0}^{+\infty}\big(\dfrac{1}{x}\big)^{\frac{\alpha-1}{2}}\exp\big(-\dfrac{\lambda\tau}{1-\tau}x\big)I_{\alpha-1}\big(\dfrac{2\lambda\sqrt{ux\tau}}{1-\tau}\big)\mathbb{P}_X(dx)\\
\end{align*}
By inequality $(6.25)$ in \cite{INEQ} which holds for $\alpha>1/2$ (for $\alpha=1/2$, we use the explicit expression of $I_{-1/2}(.)$), we have the following estimate:
\begin{align*}
I_{\alpha-1}\big(\dfrac{2\lambda\sqrt{ux\tau}}{1-\tau}\big)&\leq \dfrac{1}{2^{\alpha-1}\Gamma(\alpha)} \bigg(\dfrac{2\lambda\sqrt{ux\tau}}{1-\tau}\bigg)^{\alpha-1}\cosh\big(\dfrac{2\lambda\sqrt{ux\tau}}{1-\tau}\big),\\
&\leq C_{a,b,\lambda,\alpha}(ux)^{\frac{\alpha-1}{2}}\exp\big(\dfrac{2\lambda\sqrt{ux\tau}}{1-\tau}\big).
\end{align*}
Thus, we have:
\begin{align*}
g^{(\alpha,\lambda)}(\tau,u)&\leq C^1_{a,b,\lambda,\alpha}u^{\alpha-1}\int_{0}^{+\infty}\exp\big(-\frac{\lambda}{1-\tau}(\sqrt{u}-\sqrt{x\tau})^2\big)\mathbb{P}_X(dx),\\
&\leq C^1_{a,b,\lambda,\alpha}u^{\alpha-1}\bigg(\int_0^{\frac{u}{4\tau}}\exp\big(-\frac{\lambda}{1-\tau}(\sqrt{u}-\sqrt{x\tau})^2\big)\mathbb{P}_X(dx)\\
&+\int_{\frac{u}{4\tau}}^{+\infty}\exp\big(-\frac{\lambda}{1-\tau}(\sqrt{u}-\sqrt{x\tau})^2\big)\mathbb{P}_X(dx)\bigg),\\
&\leq C^1_{a,b,\lambda,\alpha}u^{\alpha-1}\bigg(\exp\big(-\frac{\lambda}{1-\tau}\frac{u}{4}\big)+\int_{\frac{u}{4}}^{+\infty}\mathbb{P}_X(dx)\bigg).
\end{align*}
We proceed similarly for the functions $h^{(\alpha,\lambda)}(\tau,u)$ and $k^{(\alpha,\lambda)}(\tau,u)$.
\end{proof}
\begin{lem}\label{tech3}
Let $\alpha\geq 1/2$. There exists a strictly positive constant $A_{X,\alpha,\lambda,a,b}$ such that:
\begin{align*}
\big|\log\big(g^{(\alpha,\lambda)}(\tau,u)\big)\big|\leq A_{X,\alpha,\lambda,a,b}+2\dfrac{\lambda u}{1-b}+|\alpha-1|\big|\log(u)\big|.
\end{align*}
\end{lem}
\begin{proof}
Let $(a,b)\subsetneq(0,1)$. For every $(\tau,u)\in (a,b)\times(0,+\infty)$, we have:
\begin{align*}
g^{(\alpha,\lambda)}(\tau,u)=&\dfrac{\lambda}{1-\tau}\big(\dfrac{1}{\tau}\big)^{\frac{\alpha-1}{2}}u^{\frac{\alpha-1}{2}}\exp\big(-\frac{\lambda}{1-\tau}u\big)\\
&\times\int_{0}^{+\infty}\big(\dfrac{1}{x}\big)^{\frac{\alpha-1}{2}}\exp\big(-\dfrac{\lambda\tau}{1-\tau}x\big)I_{\alpha-1}\big(\dfrac{2\lambda\sqrt{ux\tau}}{1-\tau}\big)\mathbb{P}_X(dx).\\
\end{align*}
From the previous lemma, we clearly have:
\begin{align*}
g^{(\alpha,\lambda)}(\tau,u)\leq C^1_{a,b,\lambda,\alpha}u^{\alpha-1}.
\end{align*}
Moreover, using the following inequality (see the left-hand side of inequality $(6.25)$ in \cite{INEQ}):
\begin{align*}
\forall\nu>-\frac{1}{2},\ \forall z>0,\ I_{\nu}(z)>\dfrac{z^{\nu}}{2^{\nu}}\dfrac{1}{\Gamma(\nu)},
\end{align*}
we obtain:
\begin{align*}
g^{(\alpha,\lambda)}(\tau,u)\geq C_{X,a,b,\lambda,\alpha}u^{\alpha-1}\exp\bigg(-\dfrac{\lambda u}{1-b}\bigg).
\end{align*}
Thus,
\begin{align*}
C_{X,a,b,\lambda,\alpha}u^{\alpha-1}\exp\bigg(-\dfrac{\lambda u}{1-b}\bigg)\leq g^{(\alpha,\lambda)}(\tau,u)\leq C^1_{a,b,\lambda,\alpha}u^{\alpha-1}.
\end{align*}
Consequently, we obtain:
\begin{align*}
&\big|\log\big(g^{(\alpha,\lambda)}(\tau,u)\big)\big|\\
&\leq \big|\log\big(\frac{g^{(\alpha,\lambda)}(\tau,u)}{C_{X,a,b,\lambda,\alpha}u^{\alpha-1}\exp\bigg(-\dfrac{\lambda u}{1-b}\bigg)}\big)\big|+\big|\log\big(C_{X,a,b,\lambda,\alpha}u^{\alpha-1}\exp\bigg(-\dfrac{\lambda u}{1-b}\bigg)\big)\big|,\\
&\leq \big|\log\big(\frac{C^1_{a,b,\lambda,\alpha}u^{\alpha-1}}{C_{X,a,b,\lambda,\alpha}u^{\alpha-1}\exp\bigg(-\dfrac{\lambda u}{1-b}\bigg)}\big)\big|+|\log(C_{X,a,b,\lambda,\alpha})|+|\alpha-1||\log(u)|+\dfrac{\lambda u}{1-b},\\
&\leq |\log(C^1_{a,b,\lambda,\alpha})|+2|\log(C_{X,a,b,\lambda,\alpha})|+|\alpha-1||\log(u)|+2\dfrac{\lambda u}{1-b}.
\end{align*}
The result then follows.
\end{proof}

\section*{Acknowledgements}
 
The authors thank Michel Ledoux, Ivan Nourdin and Giovanni Peccati for
fruitful discussions on  a preliminary version of this paper. We are
particularly grateful to  Michel Ledoux for his insight which led to
the results presented in Sections \ref{sec-LSI} and \ref{sec-HSI}.


\def\cprime{$'$}

\end{document}